\documentclass[review]{elsarticle}

\usepackage{a4wide,amssymb,graphics,graphicx,textcomp}
\usepackage{enumerate,textcomp,multirow,amsmath}
\usepackage{algorithm}
\usepackage{algorithmic}
\usepackage{url}
\usepackage{amsmath, amsthm}
\usepackage[colorlinks]{hyperref}
\usepackage{mathrsfs}
\numberwithin{equation}{section}

\newtheorem{theorem}{\bf Theorem}[section]

\newtheorem{definition}[theorem]{\bf Definition}
\newtheorem{corollary}[theorem]{\bf Corollary}

\newtheorem{remark}[theorem]{\bf Remark}
\newtheorem{lemma}[theorem]{\bf Lemma}
\newtheorem{example}{\bf Example}

\def\real{\mathop{\mathrm{Re}}}
\def\imag{\mathop{\mathrm{Im}}}

\newcommand {\mat}  [1] {\left[\begin{array}{#1}}
	\newcommand {\rix}      {\end{array}\right]}

\newcommand{\nrm}[1]{{\left\vert\kern-0.25ex\left\vert\kern-0.25ex\left\vert #1 
		\right\vert\kern-0.25ex\right\vert\kern-0.25ex\right\vert}}

\newcommand{\C}{{\mathbb C}}
\newcommand{\R}{{\mathbb R}}
\renewcommand{\S}{{\mathbb S}}

\journal{Linear Algebra and its Applications}

\begin{document}
	
\begin{frontmatter}
	
	%\title{On computing the structured distances to singularity for matrix pencils using optimization}
	
	\title{On structured condition number of rational matrix functions}
	
	%% Group authors per affiliation:
	\author{Ritwik Prabin Kalita}\fnref{iitd}
	\author{Anshul Prajapati\fnref{MPI}}
	\author{Punit Sharma\fnref{iitd}}
	%\author{Volker Mehrmann\fnref{tuberlin}}
	\fntext[iitd]{Department of Mathematics, Indian Institute of Technology Delhi, Hauz Khas, 110016, India; \texttt{\{maz238697, punit.sharma\}@maths.iitd.ac.in.},\texttt{ punit1718@iitdabudhabi.ac.ae.}
		%	A.P. acknowledges the support of the CSIR Ph.D. grant by Ministry of Science \& Technology, Government of India. 
		PS acknowledges the support of the SERB - CRG grant (CRG/2023/003221) and SERB-MATRICS grant by Government of India.
	}
	\fntext[MPI]{Max Planck Institute for Dynamics of Complex Technical Systems, 39106, Magdeburg, Germany;\texttt{ prajapati@mpi-magdeburg.mpg.de.} A.P. acknowledges the MPI for support through a Postdoctoral Fellowship.}
	%\fntext[tuberlin]{Institut f${\rm \ddot{u}}$r Mathematik, MA 4-5 TU Berlin, Str.\@ d.\@ 17.\@ Juni 136, D-10623 Berlin, Germany; \texttt{mehrmann@maths.tu-berlin.de.}
		%}
	
	\begin{abstract}
		We derive the necessary and sufficient conditions for the simple eigenvalues of rational matrix functions with symmetry structure to have the same normwise condition number with respect to arbitrary and structure-preserving perturbations. We obtain an exact expression for the structured condition number of simple eigenvalues of symmetric, skew-symmetric and T-even/odd rational matrix functions, and tight bounds are obtained for simple eigenvalues of Hermitian, skew-Hermitian, $*$-even/odd, $*$-palindromic and T-palindromic rational matrix functions.
		%			We derive a formula for structured backward error of complex number $\lambda$ when considered as an approximate eigenvalue of rational matrix polynomial. We consider symmetric, skew-symmetric, T-even, T-odd, Hermitian, skew-Hermitian, $*$-even, $*$-odd and $*$-palindromic structure and backward error is computed with respect to the corresponding structured perturbations.
	\end{abstract}
	
	\begin{keyword}
		matrix polynomial, rational matrix function, condition number, structured condition number.\\
		%\noindent
		{\textbf{AMS subject classification.}}
		15A18, 15A22, 65K05
		%65F20, 65F30, 90C22.
	\end{keyword}
	
\end{frontmatter}

%\linenumbers

\section{Introduction}

In this paper, we study the conditioning of simple eigenvalues of the rational eigenvalue problem (REP),
\begin{equation}
	G(z)x=0,
\end{equation}
where $G(z)$ is an $n\times n$ rational matrix function (RMF) of the form
\begin{equation}\label{mat:G}
	G(z)=\sum_{k=0}^d z^kA_k + \sum_{j=1}^m w_j(z)E_j,
\end{equation}
where $A_k$'s and $E_j$'s are from $\C^{n,n}$ that carry some symmetry structure, and $w_j(z):=\frac{s_j(z)}{q_j(z)}$, where $s_j(z)$ and $q_j(z)$ are scalar polynomials for $j=1,\ldots,m$. A scalar $\lambda\in \C$ is said to be an \emph{eigenvalue} of $G(z)$ if $G(\lambda)x=0$ or $y^*G(\lambda)=0$ for some non-zero vectors $x$ and $y$. The non-zero vectors $x$ and $y$ are called \emph{right} and \emph{left eigenvectors} corresponding to the eigenvalue $\lambda$, respectively.
When the rational part is zero in~\eqref{mat:G} (i.e., $E_j=0$ for each $j$), then it is called a matrix polynomial and we denote it by $P(z)$, i.e., $P(z)=\sum_{k=0}^d z^kA_k$. 
% In this case, $(\lambda,x,y)$ is called the \emph{eigen triplet} of matrix function $G(z)$.
This paper considers the problem of computing the structured condition number of  a simple eigenvalue $\lambda$ of RMF's $G(z)$ with structures from Table~\ref{tab:my_label}.

The REPs arise in a wide range of applications, such as in acoustic emissions of high-speed trains, loaded elastic strings, damped vibration of structures, electronic structure calculations of quantum dots, and in control theory; see~\cite{BetHMST13,MR2124762,MR2249422,MR2054343,MR2260626}. The REP with $G(z)$ of the form~\eqref{mat:G} are special cases of general nonlinear eigenvalue problems of the form
\begin{equation}
	M(z)x=0, \quad \text{ with } \quad M(z)=\sum_{j=1}^m M_j f_j(z),
\end{equation}
where $M_j$'s are $n\times k$ matrices and $f_1(z),f_2(z),\ldots, f_m(z)$ are scalar-valued functions; see~\cite{MR2124762,TisM01} for a survey on a large number of applications and~\cite{BetHMST13} for some benchmark examples. In most cases, the RMFs arising from applications follows some symmetry structure on the coefficient matrices. For instance, see the following examples.
\begin{example}{\label{examp:1}}{\rm \cite{BChoHHL11}} {\rm 
		The numerical solution of a fluid-structure interaction leads to a symmetric REP of the form 
		\[
		\left( \frac{z^2}{a^2} M + K + \frac{z^2}{z \beta + \alpha} D 
		\right)x=0,
		\]
		where $a$ is the speed of sound in the given material, $\alpha$ and $\beta$ are positive constants, and $M$ and $K$ are symmetric positive definite matrices, respectively describing the mass and the stiffness, and the matrix $D$ is symmetric positive semidefinite describing the damping effects of an absorbing wall. 
	}
\end{example} 

\begin{example}{\label{examp:2}}{\rm 
		The REP arising in the finite element simulation of mechanical problems, see~\cite{Sol06,Vos03}, have the following form:
		\[
		\left ( P(z) +Q (z) \sum_{j=1}^m \frac{z}{z-\sigma_j}E_j
		\right )x=0,
		\]
		where $P(z)$ and $Q(z)$ are symmetric matrix polynomials with large and sparse coefficients, and $E_j$'s are low-rank matrices. For a classical example, the REP
		\[
		\left(A-zB + \sum_{j=1}^m\frac{z}{z-\sigma_j}E_j\right)x=0,
		\]
		where $\sigma_j$ are positive, $A$ and $B$ are symmetric positive definite matrices of size $n\times n$ and $E_j=C_jC_j^T$ with $C_j \in \mathbb{R}^{n\times r_j}$ are of rank $r_j$, arises in the study of the simulation of mechanical vibration of fluid-solid structures~\cite{MR2249422,MR2124762}.
	}
\end{example}

The structures we consider on the weights $w_j$ and the coefficient matrices $A_k$ and $E_j$ are identical to the ones in~\cite{PraS24} and listed in Table~\ref{tab:my_label}. Note that the symmetry structures on $A_k$'s and $E_j$'s and the algebraic structures on $w_j$'s imply a symmetry structure on the RMF. The symmetry structures on the cofficient matrices of $G(z)$ force a symmetry structure on their eigen spectrum. 

As mentioned in~\cite{MR2249422,HwaLW05}, in some cases, the REPs can be transformed into a higher degree polynomial eigenvalue problem by clearing out the denominators, and the polynomial eigenvalue problem can then be linearized into a  linear eigenvalue problem. However, the size of the problem may substantially increase, and extra un-physical eigenvalues are typically introduced. These have to be recognized and removed from the computed spectrum~\cite[Example 1.3]{PraS24}.  This motivates a careful perturbation analysis of the original data that avoids converting the rational eigenvalue problem into a polynomial eigenvalue problem.

%
%Example~\ref{examp:3} shows some of the difficulties that may arise in rational eigenvalue problems and motivates a careful perturbation analysis of the original data that avoids converting the rational problem into a polynomial problem.
%
%The structures we consider on the weights $w_j$ and the coefficient matrices $A_k$ and $E_j$ are listed in Table~\ref{tab:my_label}. We observe that the symmetry structures on $A_i$'s and $E_j$'s and the algebraic structures on $w_j$'s imply a symmetry structure on the RMF. If the RMF is symmetric or skew-symmetric, then the case of both side implication holds: $G(z)$ is symmetric or skew-symmetric if and only if $A_k^T=A_k, E_j^T=E_j (A_k^T=-A_k,E_j^T=-E_j)$. On the other hand, one side of the implication is true for other structures. For example, consider the RMF
%\[
%G(z)=A_0+zA_1+\frac{z}{z-i}E_1 + \frac{z}{z+i}E_1^*,
%\]
%with $A_0^*=A_0$ and $A_1^*=A_1$, which is Hermitian, i.e., $(G(z))^*=G(\overline z)$ but the weights $w_j(z)$ do not follow the algebraic structure $w_j(z)^*=w_j(\overline{z})$. More precisely, we only consider a subclass of the corresponding structured RMF by imposing structures on the coefficient matrices $A_k$'s, $E_j$'s, and the weights $w_j$'s.
%However, the RMFs arising in most engineering applications have symmetry structures on 
%$A_k$'s, $E_j$'s, and algebraic structures on $w_j(z)$, as mentioned in Example~\ref{examp:1}.

\begin{table}[h!]
	\centering
	\begin{tabular}{|c|c|c|}\hline 
		\textbf{Structure} &  \textbf{structure on $A_k$} & \textbf{structure on $E_j$}\\ 
		\textbf{on $G(z)$}    & \textbf{$k=0,\ldots,d$} & \textbf{$j=1,\ldots,m$} \\ \hline
		symmetric  &  $A_k^T=A_k$ & $E_j^T=E_j$  \\ \hline 
		skew-symmetric & $A_k^T=-A_k$ & $E_j^T=-E_j$ \\ \hline
		
		T-even  & $A_k^T=(-1)^kA_k$ &  
		\begin{tabular}{c c}
			$E_j^T=-E_j$ ~ if $w_j(-z)=-w_j(z)$\\ $E_j^T=E_j$ ~ if $w_j(-z)=w_j(z)$   
		\end{tabular}\\ \hline
		
		T-odd & $A_k^T=(-1)^{k+1}A_k$ & 
		\begin{tabular}{c c}
			$E_j^T=-E_j$ ~if $w_j(-z)=w_j(z)$\\ $E_j^T=E_j$ ~ if $w_j(-z)=-w_j(z)$  \end{tabular}\\ \hline
		
		Hermitian & $A_k^*=A_k$ & $E_j^*=E_j$ ~ if $w_j(z)^*=w_j(\overline{z})$ \\ \hline
		skew-Hermitian & $A_k^*=-A_k$ & $E_j^*=-E_j$ ~ if $w_j(z)^*=w_j(\overline{z})$ \\ \hline
		$*$-even & $A_k^*=(-1)^pA_k$ & \begin{tabular}{c c}
			$E_j^*=E_j$ ~ if $(w_j(-z))^*=w_j(\overline{z})$\\ 
			$E_j^*=-E_j$ ~ if $(w_j(-z))^*=-w_j(\overline{z})$  
		\end{tabular} \\ \hline 
		
		$*$-odd  & $A_k^*=(-1)^{k+1}A_k$ & \begin{tabular}{c c}
			$E_j^*=E_j$ ~ if $(w_j(-z))^*=-w_j(\overline{z})$\\ 
			$E_j^*=-E_j$ ~ if $(w_j(-z))^*=w_j(\overline{z})$  
		\end{tabular} \\ \hline 
		
		$*$-palindromic & $A_k^*=A_{d-k}$ & $E_j^*=E_j$ \quad if $(w_j(z))^*={\overline z}^d w_j(\frac{1}{\overline{z}})$\\ \hline
		T-palindromic & $A_k^T=A_{d-k}$ & $E_j^T=E_j$ \quad if $w_j(z)=z^dw_j(\frac{1}{z})$\\ \hline
	\end{tabular}
	\caption{Structures on G(z)}
	\label{tab:my_label}
\end{table}

%For the classical perturbation analysis of $G(z)$, one needs to consider perturbations in the coefficient matrices $A_j$'s and $E_j$'s, and also to scalar functions $w_j$'s, i.e., consider the following perturbed rational matrix function
%\begin{equation}
%	\sum_{k=0}^d z^k(A_k+\Delta_{A_k}) + \sum_{j=1}^m \left(w_j(z)+\delta w_j(z)\right)\left(E_j+\Delta_{E_j}\right),
%\end{equation}
%and then study how the eigenvalues change under the perturbations. Unfortunately, such perturbation analysis turns out to be extremely difficult. Instead, one may assume in many applications that the perturbations in the scalar functions $w_j(z)$ are either known or can be bounded~\cite{AhmM16,Ahm19}.

Condition numbers play an important role in the sensitivity analysis of problems that compute eigen-elements of a matrix function. If the matrix function with additional symmetry structures are considered, then in that case, it is advisable to do the sensitivity analysis under structure-preserving perturbations. Eigenvalue condition number together with eigenvalue backward error bound the error in eigenvalue computation. The 
eigenvalue backward errors have been well studied for matrix polynomials with various structures~\cite{MR3194659,MR3335496,MR4404572,Tis00}. Recently, in~\cite{PraS24}, authors have studied the structured eigenvalue backward error of rational matrix functions $G(z)$ of the form~\eqref{mat:G}  with structures from Table~\ref{tab:my_label}.
The condition number for an unstructured matrix polynomial $P(z)= A_0+zA_1+\cdots z^dA_d$ was discussed in~\cite{Tis00}, which was later extended to various symmetry structures on $P(z)$ in~\cite{Bor10}. 
 In~\cite{NofNPQ23}, the condition number of \emph{unstructured} (with no symmetry structures) RMFs in the transfer function form was considered and computable formulas were obtained. However, no work has been done on the structured condition number of RMFs with symmetry structures. 
 
 Our work is mainly inspired by~\cite{Bor10} and aims at deriving computable formulas for the structured condition number of RMF's with structures from Table~\ref{tab:my_label}. We derive necessary and sufficient conditions under which the unstructured and the structured condition numbers of simple eigenvalues of RMF $G(z)$~\eqref{mat:G} with  structures from Table~\ref{tab:my_label}, are equal. Our results are generalization of the techniques used in~\cite{Bor10} for matrix polynomials with symmetry structures such as  T-palindromic, $*$-palindromic, T-alternating and $*$-alternating. We note that, this generalization is not immediate and requires a careful dealing of the weights $w_j(z):=\frac{s_j(z)}{q_j(z)}$ while perturbing the stuctured RMF~\eqref{mat:G}. As far as we know, this is the first attempt to compute structured eigenvalue condition number of RMFs.

This paper is organized as follows: In Section~\ref{sec:prelims}, we make preliminary definitions and review some results that will be used in later sections.
Inspired by~\cite{Bor10}, in Section~\ref{sec:scond}, we consider structured condition number of a simple eigenvalue of RMF's with structures from Table~\ref{tab:my_label}. All the results pertaining to the structured condition number are contained in this section. 

% In Section~\ref{sec:ucond}, we consider the unstructured condition number of a simple eigenvalue of a RMF. While in Section~\ref{sec:scond} we consider the condition number of a simple eigenvalue of RMFs with various symmetry structures as mentioned in Table~\ref{tab:my_label} and obtain necessary and sufficient conditions for the simple eigenvalues to have the same normwise condition number with respect to arbitrary and structure-preserving perturbations. We also obtain an exact expression for the structured condition number of simple eigenvalues of symmetric, skew-symmetric and T-even/odd rational matrix functions, and tight bounds are obtained for simple eigenvalues of Hermitian, skew-Hermitian, $*$-even/odd, $*$-palindromic and T-palindromic rational matrix functions.
 
 \section{Preliminaries}\label{sec:prelims}

In the following, ${\rm Herm}(n)$, ${\rm SHerm}(n)$,  ${\rm Sym}(n)$, and ${\rm Ssym}(n)$ respectively denote the sets of $n \times n$ Hermitian, skew-Hermitian, symmetric, and skew-symmetric matrices. 
By $\|\cdot\|$, we denote the spectral norm or 2-norm of a vector or a matrix. By $I_m$, we denote the identity matrix of size $m \times m$, and $i$ stands for the imaginary unit of the complex numbers, i.e.,  { $i^2=-1$}. For a complex number $z$, we define  $\text{sign}(z)=\frac{\overline{z}}{|z|}$ when $z \neq 0$ and $\text{sign}(0)=1$.

Let $G(z)$ be an $n\times n$ RMF of the form~\eqref{mat:G}. Then $G(z)$ is said to be \emph{regular} if $\text{det}(G(z))\not\equiv 0$; otherwise, it is called \emph{singular}. The triplet $(\lambda,x,y)$ with $x,y\in \C^n\setminus\{0\}$ is referred to as an \emph{eigentriplet} of $G(z)$ if $G(\lambda)x=0$ and $y^*G(\lambda)=0$. The roots of the function $q_j(z)$ are the \emph{poles} of the RMF $G(z)$, and $G(z)$ is not defined at these values. Throughout the paper, we consider $x$ and $y$ to be unit norm eigenvectors of $G(z)$ corresponding to the eigenvalue $\lambda$ and the following assumption on $\lambda$.
\begin{align}\label{assump}
	\text{\textbf{Assumption}}&:~\lambda \in \C~ \text{ be such\,that}~ \lambda~\text{is a simple eigenvalue of } G(z),\, w_j(\lambda)\neq 0~\text{for}\, j=1,\ldots,m \nonumber \\ &\hspace{0.3cm}~\text{and}~ \lambda~ \text{is not a pole of}~ G(z).
\end{align}
Motivated by the perturbation analysis in~\cite{PraS24}, we consider the structure-preserving perturbations $\Delta G(z)$ to the RMF $G(z)$ of the form
\begin{equation}\label{mat:delG}
	\Delta G(z)=\sum_{k=0}^d z^k\Delta{A_k}+\sum_{j=1}^mw_j(z)\Delta_{E_j}.
\end{equation}
In order to measure the sensitivity of $\lambda$, we first need to fix a norm to measure the perturbations in RMF $\Delta G(z)$. Given a matrix norm $\|\cdot\|_M$ on $\C^{n,n}$, and a vector norm $\|\cdot\|_v$ on $\C^{d+m+1}$, %we introduce the following norm on $(C^{n,n})^{d+m+1}$. For a tuple of matrices $(\Delta_{A_{0}}, \ldots, \Delta_{A_{d}},\Delta_{E_1},\ldots, \Delta_{E_k}) \in (\mathbb{C}^{n \times n})^{d+m+1}$,
we define 
\begin{equation}
	\|\Delta G(z)\|_{M,v} = %\nrm{\left(\Delta_{A_{0}}, \ldots, \Delta_{A_{d}},\Delta_{E_1}, \ldots, \Delta_{E_k}\right)}:=
	{\big \|{\left({\|\Delta_{A_0}\|}_M,\ldots,{\|\Delta_{A_d}\|}_M,{\|\Delta_{E_1}\|}_M,\ldots,{\|\Delta_{E_m}\|}_M \right)\big \|}_v}.
\end{equation}
We use the norm ${\|\Delta G(z)\|}_{2,\infty}$ to measure the perturbations in $G(z)$. Thus, the norm of the perturbation $\Delta G(z)$ of the form~\eqref{mat:delG}, is given by
\begin{equation*}
	\|\Delta G(z)\|_{2,\infty} = \max \{\|\Delta_{A_0}\|,\ldots,\|\Delta_{A_d}\|,\|\Delta_{E_1}\|,\ldots,\|\Delta_{E_m}\|\}.
\end{equation*}

Note that, for computing the condition number of simple eigenvalues of matrix polynomial $P(z)$, various norms have been used in literature to measure the perturbations in $P(z)$. For example, in~\cite{Tis00,Bor10}, the norm ${\|\cdot\|}_{2,\infty}$ was used, while in~\cite{AdhAK11}, the $\|\cdot\|_{2,2}$ and $\|\cdot\|_{F,2}$ were used in order to measure the perturbations in $P(z)$, where $F$ denotes the Frobenius norm on the set of $n\times n$ matrices.

\begin{remark}{\rm
		We note that a remark similar to~\cite[Remark 2.1]{PraS24} also holds on the assumption that $w_{j}(\lambda)\neq 0$ for $j=0,\ldots,m$. In fact, the general theory does not necessarily require this assumption on the weights $w_{j}$'s. However, for the sake of simplicity and uniform notation, we assume that $w_{j}(\lambda)\neq 0$ for $j=0,\ldots,m$.
	}
\end{remark}

We are now in a position to define the absolute condition number of a simple eigenvalue of $G(z)$~\eqref{mat:G} with respect to the norm ${\|\cdot\|}_{2,\infty}$. 

\begin{definition}
	Given a RMF $G(z)$ of the form~\eqref{mat:G}, and a simple eigenvalue $\lambda\in \C$, let $x$ be a right eigenvector of unit norm and $y$ be a corresponding left eigenvector of unit norm so that $G(\lambda)x=0$ and $y^*G(\lambda)=0$. Further, let 
	$\mathbb S \subseteq {(\C^{n,n})}^{d+m+1}$. Then 
\begin{align}\label{def:ucond}
	\kappa^{\mathbb S}(\lambda,G):&=\lim_{\epsilon\to 0} \sup \bigg\{ \frac{|\Delta \lambda|}{\epsilon} : (G(\lambda +\Delta \lambda) + \Delta G(\lambda +\Delta \lambda))(x+\Delta x) = 0, \nonumber\\ 
	&\hspace{1.5 cm} (\Delta_{A_0},\ldots,\Delta_{A_d},\Delta_{E_1},\ldots,\Delta_{E_m}) \in \mathbb S, \,\Delta G(z)=\sum_{k=0}^d z^k\Delta_{A_k}+\sum_{j=1}^mw_j(z)\Delta_{E_j},\nonumber \\ 
	&\hspace{2cm}  \|\Delta_{A_k}\|\leq \epsilon,~\|\Delta_{E_j}\|\leq \epsilon~~k=0,\ldots,d,~j=1,\ldots,m  \bigg\}
\end{align}
is called the \emph{structured condition number} of $\lambda$ with respect to $G(\lambda)$ and $\mathbb S$.
\end{definition}
When $\mathbb S= {(\C^{n,n})}^{d+m+1}$ in~\eqref{def:ucond}, then it is called the \emph{unstructured condition number} and is denoted by $\kappa(\lambda,G)$. In case $\lambda=\infty$, $\kappa^{\mathbb S}(\infty,G)$ is obtained by replacing $\frac{|\Delta \lambda|}{\epsilon}$ by $\frac{1}{\epsilon|\Delta \lambda|}$ in~\eqref{def:ucond}.

Expanding the constraint $(G(\lambda +\Delta \lambda) + \Delta G(\lambda +\Delta \lambda))(x+\Delta x) = 0$ upto first-order terms and premultiplying by $y^*$, we have
\begin{equation}\label{eq:dellambda}
	\Delta \lambda = - \frac{y^*\Delta G(\lambda) x}{y^*G'(\lambda)x} + \mathcal O(\epsilon^2).
\end{equation}
Note that, $y^*G'(\lambda)x \neq 0$, as $\lambda$ is a simple eigenvalue of $G(z)$~\cite[Theorem 3.2]{AndCL93}. Thus $\kappa^{\mathbb S}(\lambda,G)$ becomes 
\begin{align}\label{eq:scond1}
	\kappa^{\mathbb S}(\lambda,G):&=\lim_{\epsilon\to 0} \sup \bigg\{\frac{|y^*\Delta G(\lambda)x|}{\epsilon |y^*G'(\lambda)x|} : (\Delta_{A_0},\ldots,\Delta_{A_d},\Delta_{E_1},\ldots,\Delta_{E_m}) \in \mathbb S, \nonumber\\ 
	&\hspace{1.5 cm}  \,\Delta G(z)=\sum_{k=0}^d z^k\Delta_{A_k}+\sum_{j=1}^mw_j(z)\Delta_{E_j},~ \|\Delta_{A_k}\|\leq \epsilon,~\|\Delta_{E_j}\|\leq \epsilon  \bigg\}.
\end{align}
For the matrix polynomial case, that is, when $G(z)=P(z)=A_0+zA_1+\cdots+z^dA_d$, the unstructured condition number of a simple eigenvalue was considered in~\cite{Tis00} and the following expression was obtained:
\begin{equation}
	\kappa(\lambda,P)= \frac{\sum_{k=0}^d |\lambda|^k}{|y^*P'(\lambda)x|}.
\end{equation}
A similar expression can be obtained for the condition number $\kappa(\lambda,G)$ for a RMF of the form~\eqref{mat:G}, defined in~\eqref{def:ucond}. To see this, from~\eqref{eq:scond1}, we have
\begin{align}\label{eq:unst2}
	\kappa(\lambda,G)
	&\leq \lim_{\epsilon\to 0} \sup \bigg\{ \frac{\left(\sum_{k=0}^d |\lambda|^k \|\Delta_{A_k}\| + \sum_{j=1}^m |w_j(\lambda)| \|\Delta_{E_j}\|\right)\|x\|\|y\|}{\epsilon |y^*G'(\lambda)x|} : \|\Delta_{A_k}\|\leq \epsilon, \|\Delta_{E_j}\|\leq \epsilon %& \hspace{5cm} \|\Delta_{E_j}\|\leq \epsilon ,~k=0,\ldots,d,~ j=1,2,\ldots,m 
	\bigg\}\nonumber \\
%	&\leq \lim_{\epsilon\to 0} \sup \left\{ \frac{\alpha }{|y^*G'(\lambda)x|} \right\} \nonumber \\
	&\leq \frac{\alpha }{|y^*G'(\lambda)x|},
\end{align}
where $\alpha = \sum_{k=0}^d |\lambda|^k + \sum_{j=1}^m |w_j(\lambda)|$, and $x$ and $y$ are unit norm vectors. This gives an upper bound to $\kappa(\lambda,G)$, which is attained by perturbations, $\Delta_{A_k} = -\epsilon \text{sign}(\lambda^k)M$ for $k=0,\ldots, d$ and $\Delta_{E_j} = -\epsilon \text{sign}(w_j(\lambda)) M$ for $j=1,2,\ldots, m$ with $M=yx^*$. Note that $\|M\|=1$ and $y^*Mx = 1$. Also,
\begin{align*}
	y^*\Delta G(\lambda)x &= \sum_{k=0}^d \lambda^k (-\epsilon \text{sign}(\lambda^k)) + \sum_{j=1}^m w_j(\lambda)(-\epsilon \text{sign}(w_j(\lambda)))= -\epsilon \alpha ,
\end{align*}
which gives $|y^*\Delta G(\lambda)x| = \epsilon \alpha $, implying the equality in~\eqref{eq:unst2}. In summary, we have the following result. 

\begin{theorem}\label{thm:unstrcond}
	Let $G(z)$ be a RMF of the form~\eqref{mat:G} and let $(\lambda,x,y)$ be an eigentriplet of $G(z)$ such that $\lambda$ satisfy the assumption~\eqref{assump}, and $\|x\|=1$ and $\|y\|=1$. Then
	\begin{equation*}
		\kappa(\lambda,G) = \frac{\alpha }{|y^*G'(\lambda)x|},
	\end{equation*}
	where $\alpha = \sum_{k=0}^d |\lambda|^k + \sum_{j=1}^m |w_j(\lambda)|$. 
\end{theorem}
We close this section with two mapping results that will be useful in computing structured condition number $	\kappa^{\mathbb S}(\lambda,G)$. 

%\newpage

\begin{lemma}{\rm \cite{MacMT08}}\label{lem:map}
	Let $S \subseteq \C^{n,n}$ and vectors $v,u \in \C^n$ with $v \neq 0$.  Table~\ref{tab2map} gives necessary and suﬃcient conditions for the existence of $\Delta\in S$ such that $\Delta v=u$.
	\begin{table}[h!]
		\centering
		\begin{tabular}{|l|l|}\hline
			$S$ & condition on $v$ and $u$ \\ \hline
			$\C^{n,n}$ & none \\
			Symmetric & none \\ 
			Skew-symmetric & $v^Tu=0$ \\ 
			Hermitian & $v^*u \in \R$ \\ 
			Skew-Hermitian & $v^*u \in i\R$ \\ \hline
		\end{tabular}
		\caption{Structured mapping conditions}
	\end{table}\label{tab2map}
\end{lemma}

\begin{lemma}{\rm \cite{MehMS17}}\label{lem:map-pal}
	Let $u,v,w \in \C^n$ with $v\neq 0$, and define
	\begin{align*}
		S_1:=\{\Delta \in \C^{n,n} ~:~ \Delta v = u, \Delta^*v=w\}, \;\;
		S_2:=\{\Delta \in \C^{n,n} ~:~ \Delta v = u, \Delta^Tv=w\}.
	\end{align*}
	Then, $S_1\not = \emptyset$ if and only if $v^*u=w^*v$ and $S_2\not = \emptyset$ if and only if $v^Tu=w^Tv$.
\end{lemma}

\section{Structured condition number}\label{sec:scond}

As mentioned in~\cite{Bor10}, if the matrix function $G(z)$ is structured,  that is, the coefficient matrices $A_k$'s and $E_j$'s, and the scalar functions $w_j(z)$'s follow symmetry structures from Table~\ref{tab:my_label}, then it is of practical interest to understand the sensitivity of $\lambda$ with respect to structure-preserving perturbations. In this section, we consider the structured condition number $\kappa^{\mathbb S}(\lambda,G)$ of a simple eigenvalue $\lambda$ of $G(z)$ with structures from~Table~\ref{tab:my_label}. 

We aim at deriving necessary and sufficient conditions such that the unstructured and the structured condition number for a given simple eigenvalue $\lambda$ are equal, that is, $\kappa(\lambda,G) = \kappa^{\mathbb{S}}(\lambda,G)$. For this, we give the following lemma, which is a generalized version of~\cite[Lemma 3.2]{Bor10} for the rational matrix functions and will be useful in characterizing the simple eigenvalues of $G(z)$ with $\kappa(\lambda,G) = \kappa^{\mathbb{S}}(\lambda,G)$ for various structures under consideration.
\begin{lemma}\label{lem:SequalU}
	Let $\mathbb S \subseteq {(\C^{n,n})}^{d+m+1}$ and let $G(z)$ be a RMF of the form~\eqref{mat:G}.
	%with $(\Delta_{A_0},\ldots,\Delta_{A_d},\Delta_{E_1},\ldots,\Delta_{E_m}) \in \mathbb S$. 
	Let $(\lambda,x,y)$ be an eigentriplet of $G(z)$ such that $\lambda$ satisfy the assumption~\eqref{assump}, and $\|x\|=1$ and $\|y\|=1$. If $\lambda \neq 0$, then $\kappa^{\S}(\lambda, G) = \kappa(\lambda, G)$ if and only if there exists $(\Delta_{A_0},\ldots,\Delta_{A_d},\Delta_{E_1},\ldots, \Delta_{E_m}) \in {\S}$ such that $\Delta_{A_k} x = \omega \text{sign}(\lambda^k) y$ for $k=0,\ldots,d$ and $\Delta_{E_j} x = \omega \text{sign}(w_j(\lambda)) y$ for $j=1,\ldots,m$, where $\omega \in \mathbb{C}$ such that $|\omega| = 1$. 
	
	If $\lambda=0$, then $\kappa^{\S}(0, G) = \kappa(0, G)$ if and only if $(\Delta_{A_0},\ldots,\Delta_{A_k},\Delta_{E_1},\ldots,\Delta_{E_m})\in \S$ such that $\Delta_{A_0}x=\omega y$ and $\Delta_{E_j} x = \omega \text{sign}(w_j(\lambda)) y$ for $j=1,\ldots,m$.
\end{lemma}
\begin{proof}
In view of~\eqref{eq:scond1}, we have 
	\begin{equation}
		\kappa^\mathbb{S}(\lambda, G) = \limsup_{\epsilon \to 0} \frac{|y^\ast \Delta G(\lambda) x|}{\epsilon |y^\ast G'(\lambda) x|},
	\end{equation}
	where $\Delta G(\lambda) := \sum_{k=0}^{d} \lambda^k \Delta_{A_k} + \sum_{j=1}^{m} w_j(\lambda)\Delta_{E_j} $ is such that $(\Delta_{A_0},\ldots,\Delta_{A_d},\Delta_{E_1},\ldots, \Delta_{E_m}) \in {\S}$ and $\|\Delta_{A_k}\| \leq \epsilon$ for $k = 0,\ldots, d$ and $\|\Delta_{E_j}\| \leq \epsilon$ for $j=1,\ldots,m$.
	By taking $\alpha_k = \frac{y^*\Delta_{A_k} x}{\epsilon}$ and $\beta_j = \frac{y^*\Delta_{E_j} x}{\epsilon}$, we have
	\begin{equation}\label{eq:firsteq}
		\kappa^\S(\lambda, G) = \limsup_{\epsilon \to 0} \frac{\big |\sum_{k=0}^{d} \alpha_k \lambda^k+ \sum_{j=1}^{m} \beta_jw_j(\lambda)\big|}{ |y^\ast G'(\lambda) x|}.
	\end{equation}
	Note that $|\alpha_k| = \big|\frac{y^*\Delta_{A_k} x}{\epsilon}\big | \leq \|y\| \|x\| = 1$, as $\|\Delta_{A_k}\|\leq \epsilon$ and $x,y$ are unit norm vectors. Similarly, we have $|\beta_j| \leq 1$. This implies from~\eqref{eq:firsteq} that 
	\[
	\kappa^\S(\lambda, G) \leq \frac{\sum_{k=0}^{d}  |\lambda^k|+ \sum_{j=1}^{m} |w_j(\lambda)|}{ |y^\ast G'(\lambda) x|} = \kappa(\lambda, G).
	\]
Thus, equality of $\kappa^\S(\lambda, G)$ and $\kappa(\lambda, G)$ holds if and only if there exists \(\Delta G(\lambda)\) with coefficient matrices 
$(\Delta_{A_0},\ldots,\Delta_{A_d},\Delta_{E_1},\ldots, \Delta_{E_m}) \in {\S}$, such that 
\(\|\Delta_{A_k}\| \leq 1\) for \(k = 0,1,\ldots,d\),  \(\|\Delta_{E_j}\| \leq 1\) for \(j = 1,\ldots, m\) and  \(|y^\ast \Delta G(\lambda) x| = \sum_{k=0}^{d} | \lambda |^k\)+$\sum_{j=1}^{m} | w_j(\lambda) |$. This further reduces to, 
	\begin{align*}
		\sum_{k=0}^{d} | \lambda |^k+\sum_{j=1}^{m} | w_j(\lambda) | =|y^\ast \Delta G(\lambda) x| & = \big |\sum_{k=0}^{d} \lambda^k y^*\Delta_{A_k} \, x+ \sum_{j=1}^{m} y^* \Delta_{E_j} \, x \,  w_j(\lambda)\big| \\
		& \leq \sum_{k=0}^{d} |\lambda^k| |y^*\Delta_{A_k} \, x| + \sum_{j=1}^{m} |y^* \Delta_{E_j} \, x|| w_j(\lambda)| \\
		& \leq \sum_{k=0}^{d} |\lambda|^k + \sum_{j=1}^{m} | w_j(\lambda)|,
	\end{align*}
	which implies that equality must hold at each step, i.e.,
	\begin{eqnarray*}
	|y^* \Delta_{A_k} x| = 1~~\text{for}~k=0,\ldots,d,  \quad |y^* \Delta_{E_j} x|=1,~\text{for}~~j=1,\ldots,m,
	\end{eqnarray*}
and
\begin{equation*}
		\big |\sum_{k=0}^{d} \lambda^k y^*\Delta_{A_k} \, x+ \sum_{j=1}^{m} y^* \Delta_{E_j} \, x\, w_j(\lambda)\big| = \sum_{k=0}^{d} |\lambda^k y^*\Delta_{A_k} \, x| + \sum_{j=1}^{m} |w_j(\lambda) y^* \Delta_{E_j} \, x |.
\end{equation*}
	This is possible if and only if \(\lambda^k y^* \Delta_{A_k} x\) for \(k = 0,\ldots, d\) and  \(w_j(\lambda) y^\ast \Delta_{E_j} x\) for \(j = 1, \ldots, m\) are collinear, that is, 
	\begin{equation}\label{eq:temppu1}
		y^\ast \Delta_{A_k} x \lambda^k = a_k\,\omega \text{ and } y^\ast \Delta_{E_j}\, x\, w_j(\lambda)=e_j \, \omega,
	\end{equation}
for some $\omega\in \C$ with $|\omega|=1$ and $a_k,e_j>0$. From~\eqref{eq:temppu1}, we have $a_k={|\lambda|}^k$ and $e_j=|w_j(\lambda)|$, and thus 
\begin{equation}\label{eq:temppu2}
		y^\ast \Delta_{A_k} x= \omega \frac{|\lambda|^k}{\lambda^k}=\omega \text{sign}(\lambda^k) \quad \text{and}\quad 
			y^\ast \Delta_{E_j} x= \omega \frac{|w_j(\lambda)|}{w_j(\lambda)}= \omega \text{sign}(w_j(\lambda)).
\end{equation}
If $\lambda\neq 0$, using Cauchy-Schwarz inequality,~\eqref{eq:temppu2} can be further  reduced to 
	\begin{equation*}
		\Delta_{A_k} x = \omega \text{sign}((\lambda)^k) y \text{ and } \Delta_{E_j} x = \omega \text{sign}(w_j(\lambda)) y.
	\end{equation*}
	If $\lambda=0$, then proceeding on similar lines, we obtain $\kappa^{\S}(0,G)=\kappa(0,G)$ if and only if $\Delta_{A_0} x = \omega y \text{ and } \Delta_{E_j} x = \omega \text{sign}(w_j(\lambda)) y$.
\end{proof}
In view of~\eqref{eq:firsteq}, finding expressions for $\kappa^{\S}(\lambda,G)$ is equivalent to maximizing $\big|\sum_{k=0}^d \lambda^k a_k + \sum_{j=1}^m w_j(\lambda) b_j\big |$ as each $a_k$ and $b_j$ varies over compact sets $\mathcal M_S(x, y) := \{y^*Ax : A \in S,~ \|A\| \leq 1\}$, where $S\subseteq \C^{n,n}$ denotes Hermitian, skew-Hermitian, symmetric, or skew-symmetric structure. The following theorem from~\cite{Kar10} will be  used to estimate $\kappa^{\S}(\lambda,G)$ for structures from Table~\ref{tab:my_label}. 
\begin{theorem}\label{thm:setK}
	Given $A\in S\subseteq \C^{n,n}$ such that $\|A\|\leq 1$ and $x,y$ be unit norm vectors in $\C^{n,n}$, Then, the set $\mathcal M_S(x,y)$ is given by
	\begin{equation*}
		\mathcal M_S(x,y): = \left\{ e^{i\phi}(a\alpha + i b\beta) ~:~ \alpha^2+\beta^2\leq 1 \right\},
	\end{equation*} 
	where $\phi, a$ and $b$ are given by 
	\begin{table}[h!]
		\centering
		\begin{tabular}{|l|c|c|c|} \hline
			$S$ & $a$ & $b$ & $\phi$ \\ \hline 
			$\C^{n,n}$ & 1 & 1 & 0 \\ \hline
			Hermitian & 1 & $\sqrt{1-|y^*x|^2}$ & arg$(y^*x)$\\ \hline
			Skew-Hermitian & $\sqrt{1-|y^*x|^2}$ & 1 & arg$(y^*x)$\\ \hline
			Complex symmetric & 1 & 1 & 0 \\ \hline
			Complex skew-symmetric & $\sqrt{1-|y^Tx|^2}$ & $\sqrt{1-|y^Tx|^2}$ & 0 \\ \hline
		\end{tabular}
	\end{table}
\end{theorem}

\subsection{Symmetric and skew-symmetric structures} \label{sec:sym}

In this section, we examine RMFs $G(z)$ in the form~\eqref{mat:G} that are either symmetric or skew-symmetric and satisfy $(G(z))^T=G(z)$ or $(G(z))^T=-G(z)$, respectively. Consequently, the coefficient matrices $A_k$'s and $E_j$'s are all symmetric when $G(z)$ is symmetric and skew-symmetric when $G(z)$ is skew-symmetric, as illustrated in Table~\ref{tab:my_label}. The structured condition number is denoted by $\kappa^{{\rm sym}}(\lambda,G)$ when $\mathbb S ={\rm Sym(n)}^{d+m+1}$ and $\kappa^{{\rm ssym}}(\lambda,G)$ when $\mathbb S ={\rm Ssym(n)}^{d+m+1}$ in~\eqref{def:ucond}.
\begin{theorem}
	Let $G(z)$ be a symmetric RMF of the form~\eqref{mat:G}, and let $(\lambda,x,y)$ be an eigentriplet of $G(z)$ such that $\lambda$ satisfy the assumption~\eqref{assump}, and $\|x\|=1$ and $\|y\|=1$. Then, $$\kappa^{\rm sym}(\lambda,G) = \kappa(\lambda,G).$$
\end{theorem}
\begin{proof}\sloppy
	If $\lambda\neq 0$, then by Lemma~\ref{lem:SequalU}, $\kappa^{{\rm sym}}(\lambda,G)=\kappa(\lambda,G)$ if and only if there exists $(\Delta_{A_0},\ldots,\Delta_{A_d},\Delta_{E_1},\ldots,\Delta_{E_m})\in (\text{Sym}(n))^{d+m+1}$
	satisfying $\Delta_{A_k} x = \omega \text{sign}((\lambda)^k) y \text{ and } \Delta_{E_j} x = \omega \text{sign}(w_j(\lambda)) y$. In view of Lemma~\ref{lem:map}, there always exist such symmetric matrices $\Delta_{A_k}$ and $\Delta_{E_j}$. The case $\kappa^{\rm sym}(0,G) = \kappa(0,G)$ follows on the same lines.
\end{proof}

In the following result, we obtain structured condition number $\kappa^{\rm ssym}(\lambda,G)$, when $G(z)$ is skew-symmetric. 
\begin{theorem}
	Let $G(z)$ be a skew-symmetric RMF and let $(\lambda,x,y)$ be an eigentriplet of $G(z)$ such that $\lambda$ satisfy the assumption~\eqref{assump}, and $\|x\|=1$ and $\|y\|=1$. Then
\begin{equation}\label{eq:skew}
	\kappa^{\rm ssym}(\lambda,G) = c \cdot \kappa(\lambda,G),
\end{equation}
where $c:=\sqrt{1-|x^Ty|^2}$. In particular, 
$\kappa^{\rm ssym}(\lambda,G) = \kappa(\lambda,G)$ if and only if $x^Ty=0$. 
\end{theorem}
\begin{proof}\sloppy
%	Let $\lambda \neq 0$, then by Lemma~\ref{lem:SequalU}, $\kappa^{\text{ssym}}(\lambda,G)=\kappa(\lambda,G)$ if and only if there exists $(\Delta_{A_0},\ldots,\Delta_{A_d},\Delta_{E_1},\ldots,\Delta_{E_m})\in (\text{Ssym}(n))^{d+m+1}$ satisfying
%	%
%	\begin{equation}\label{skew:map}
%		\Delta_{A_k} x = \omega \text{sign}((\lambda)^k) y \text{ and } \Delta_{E_j} x = \omega \text{sign}(w_j(\lambda)) y.
%	\end{equation}
%	%
%	Using Lemma~\ref{lem:map}, there exists a skew-symmetric matrix $\Delta\in \C^{n,n}$ sending a vector $a\in \C^n$ to another vector $b\in \C^n$ if and only if $a^Tb=0$. Using this in~\eqref{skew:map}, there exists skew-symmetric $\Delta_{A_k}$ and $\Delta_{E_j}$ satisfying~\eqref{skew:map} if and only if $x^T(w\text{sign}(\lambda^k)y)=0$ for $k=0,1,\ldots, d$ and $x^T(w\text{sign}(w_j(\lambda))y)=0$ for $j=1,\ldots, m$, which collectively gives $x^Ty=0$. The proof is similar for the case $\lambda=0$. This completes the first part.
	
We first show that $\kappa^{\rm ssym}(\lambda,G) \leq c\cdot \kappa(\lambda,G)$. For this, in view of~\eqref{eq:scond1} and Theorem~\ref{thm:unstrcond}, it is enough to show that 
	\begin{equation}
		\sup \big \{ \big |y^*( \sum_{k=0}^d \lambda^k \Delta_{A_k} + \sum_{j=1}^m w_j(\lambda) \Delta_{E_j} )x\big | \big \} \leq c\cdot \alpha,
	\end{equation}
	where supremum is taken over skew-symmetric matrices $\Delta_{A_k}$ and $\Delta_{E_j}$ such that $\|\Delta_{A_k}\|\leq 1, \|\Delta_{E_j}\|\leq 1$ for $k=0,\ldots,d$,  $j=1,\ldots, m$, 
	%$c=\left(\sqrt{1-|x^Ty|^2}\right)$, 
	and $\alpha=\sum_{k=0}^{d} |\lambda|^k + \sum_{j=1}^{m} | w_j(\lambda)|$. In view of Theorem~\ref{thm:setK}, this is equivalent to show that 
	\begin{equation}
		\sup_{\alpha_k,\beta_k,\delta_j,\gamma_j\in \R} \left \{ \big | \sum_{k=0}^d \lambda^k c(\alpha_k+i\beta_k) + \sum_{j=1}^m w_j(\lambda) c(\delta_j + i\gamma_j)\big | ~:~ \alpha_k^2+\beta_k^2\leq 1, ~\delta_j^2 + \gamma_j^2\leq 1\right \} \leq c \cdot \alpha,
	\end{equation}
which follows immediately by using the triangular inequality as 
	\begin{align*}
		\Big | \sum_{k=0}^d \lambda^k c(\alpha_k+i\beta_k) + \sum_{j=1}^m w_j(\lambda) c(\delta_j + i\gamma_j)\Big | &\leq c\Big (\sum_{k=0}^d |\lambda|^k|\alpha_k+i\beta_k| + \sum_{j=1}^m |w_j(\lambda)| |\delta_j + i\gamma_j|\Big )\\
		&\leq c\alpha.
	\end{align*}
This proves the inequality in~\eqref{eq:skew}. The equality in~\eqref{eq:skew} holds, if we  show that there exist skew-symmetric matrices $\hat \Delta_{A_k}$ and $\hat \Delta_{E_j}$ such that $\|\hat \Delta_{A_k}\|\leq 1, \|\hat \Delta_{E_j}\|\leq 1$ and 
	\begin{equation*}
		\Big|y^*( \sum_{k=0}^d \lambda^k \hat \Delta_{A_k} + \sum_{j=1}^m w_j(\lambda) \hat \Delta_{E_j} )x\Big|=c\alpha.
	\end{equation*}
	Define $\hat \Delta_{A_k}=\text{sign}(\lambda^k)R$ for $k=0,\ldots, d$ and $\hat\Delta_{E_j}=\text{sign}(w_j(\lambda))R$, where $R$ is a unit 2-norm skew-symmetric matrix satisfying $y^*Rx=c$. Note that such a choice of $R$ is possible by Theorem~\ref{thm:setK}. This implies that 
	\begin{align*}
		\Big |y^*( \sum_{k=0}^d \lambda^k \hat \Delta_{A_k} + \sum_{j=1}^m w_j(\lambda) \hat \Delta_{E_j} )x \Big| &= \Big | \sum_{k=0}^d \lambda^k y^*\hat\Delta_{A_k} x + \sum_{j=1}^m w_j(\lambda) y^* \hat \Delta_{E_j} x\Big |\\
		&= c\Big(\sum_{k=0}^d |\lambda|^k + \sum_{j=1}^m |w_j(\lambda)|\Big).
	\end{align*}
	and hence the proof.
\end{proof}

\subsection{Hermitian and related structures}
In this section, we consider RMFs $G(z)$ of the form~\eqref{mat:G} with Hermitian and related structures like skew-Hermitian, $*$-even, and $*$-odd. Recall, from Table~\ref{tab:my_label} that $G(z)$ is Hermitian if the coefficient matrices $A_k$'s and $E_j$'s in  $G(z)$ are all Hermitian and the weight functions satisfy that $(w_j(z))^*=w_j(\overline z)$ for $j=1,\ldots,m$. This implies that $(G(z))^*=G(\overline z)$.  In this case, the corresponding structured condition number is denoted by $\kappa^{{\rm Herm}}(\lambda,G)$ and defined by~\eqref{def:ucond} when $\mathbb S=\left({\rm Herm}(n)\right)^{d+m+1}$.

As shown in~\cite{Tis00} for Hermitian matrix polynomials, when $\lambda \in \R$, it is easy to check that there is no difference between the structured and the unstructured condition number for RMF as well. This is obtained in the following theorem, which also gives a lower bound to the structured condition number for $\lambda \in \C\setminus \R$.
\begin{theorem}\label{thm:herm}
	Let $G(z)$ be a Hermitian RMF of the form~\eqref{mat:G} and let $(\lambda,x,y)$ be an eigentriplet of $G(z)$ such that $\lambda$ satisfy the assumption~\eqref{assump}, and $\|x\|=1$ and $\|y\|=1$. If $\lambda\neq 0$, then $\kappa^{\rm Herm}(\lambda,G) = \kappa(\lambda,G)$ if and only if $x^*y=0$ or $\lambda\in \R$.
Also, $\kappa^{\text{Herm}}(0,G) = \kappa(0,G)$ if and only if $x^*y=0$ or $w_j(0) \in \R$ for all $j=1,\ldots, m$. Moreover, when $\lambda \in \C \setminus \R$, we have
	\begin{equation}\label{eq:Herm}
		\frac{1}{\sqrt{2}}\kappa(\lambda,G) \leq \kappa^{\rm Herm}(\lambda,G).  
	\end{equation}
\end{theorem}
\begin{proof}\sloppy
	First suppose that $\lambda\neq 0$, then by Lemma~\ref{lem:SequalU}, $\kappa^{\text{Herm}}(\lambda,G) = \kappa(\lambda,G)$ if and only if there exists $(\Delta_{A_0},\ldots,\Delta_{A_d},\Delta_{E_1},\ldots,\Delta_{E_m})\in (\text{Herm}(n))^{d+m+1}$ such that $\|\Delta_{A_k}\|\leq 1$ and $\|\Delta_{E_j}\|\leq 1$ satisfying
	\begin{align}\label{eq:herm:map}
		\Delta_{A_k}x  = \omega \text{sign}(\lambda^k)y,~ \text{ for } k=0,\ldots,d \quad \text{and}\quad 
		\Delta_{E_j}x  = \omega \text{sign}(w_j(\lambda))y~ \text{ for } j=1,\ldots,m,
	\end{align}
	for some $\omega \in \C$ such that $|\omega|=1$.
In view of Lemma~\ref{lem:map}, there exist Hermitian matrices $\Delta_{A_k}$'s and $\Delta_{E_j}$'s satisfying~\eqref{eq:herm:map} if and only if
	\begin{align}\label{eq:equiconda}
		\omega \text{sign}(\lambda^k) x^*y \in \R,~ \text{ for } k=0,1,\ldots,d 
		\quad \text{and}\quad 
		\omega \text{sign}(w_j(\lambda)) x^*y \in \R,~ \text{ for } j=1,\ldots,m .
	\end{align}
	For each $k$, the constraint $\omega \text{sign}(\lambda^k) x^*y \in \R$ is equivalent to $\text{Im}(\omega \text{sign}(\lambda^k) x^*y)=0$. Using polar coordinates, we have
	\begin{equation*}
		\text{Im}(\omega \text{sign}(\lambda^k) x^*y)=0 \iff x^*y=0\quad  \text{  or  }\quad  \sin(\arg(\omega\text{sign}(\lambda^k)x^*y))=0.
	\end{equation*}
	The later condition $\sin(\arg(\omega\text{sign}(\lambda^k)x^*y))=0$ is equivalent to 
	\begin{equation*}
		\arg(\omega) - k \arg(\lambda) + \arg(x^*y) \in \pi \mathbb{Z}.
	\end{equation*}
	On comparing $\arg(w)$ for all $k=0,1,\ldots, d$, it is easy to verify that the above equation holds if and only if $\arg(\lambda)$ is an integral multiple of $\pi$, that is, $\lambda$ is real. Similarly, for $j=1,\ldots, m$, the other constraint $\omega \text{sign}(w_j(\lambda)) x^*y \in \R$ in~\eqref{eq:equiconda} holds if and only if $x^*y=0$ or 
	\begin{equation}\label{eq:herm:arg}
		\sin(\arg(\omega) - \arg(w_j(\lambda)) + \arg(x^*y)) = 0,
	\end{equation}
since $\overline{w_j(z)} = w_j(\overline{z})$ for all $j=1,\ldots, m$, as $G(z)$ is a Hermitian RMF. Also, if we choose $\omega$ such that $\arg(\omega)=n\pi - \arg(x^*y)$, then~\eqref{eq:herm:arg} holds for all $j=1,\ldots, m$ if and only if $\lambda \in \R$. 
This shows that for $\lambda\neq 0$, $\kappa^{\text{Herm}}(\lambda,G) = \kappa(\lambda,G)$ if and only if $x^*y=0$ or $\lambda\in \R$. 

If $\lambda = 0$, then by Lemma~\ref{lem:SequalU}, $\kappa^{\rm Herm}(0,G) = \kappa(0,G)$ if and only if there exist Hermitian matrices $\Delta_{A_0}$, $\Delta_{E_1},\ldots,\Delta_{E_m}$ such that $\Delta_{A_0} x = \omega y$ and $\Delta_{E_j}x = \omega \text{sign}(w_j(\lambda))y, \text{ for } j=1,\ldots,m$. Again by Lemma~\ref{lem:map}, such matrices exist if and only if $\omega x^*y\in \R$ and $\omega \text{sign}(w_j(\lambda)) x^*y \in \R$, which collectively holds if and only if $x^*y=0$ or
	\begin{equation}\label{eq:herm:arg1}
		\sin(\arg(\omega) + \arg(x^*y))=0, \text{ and } \sin(\arg(\omega) - \arg(w_j(0)) + \arg(x^*y)) = 0.
	\end{equation}
	An easy calculation shows that the~\eqref{eq:herm:arg1} holds if and only if $w_j(0)\in \R$ for $j=1,\ldots,m$ with a choice of $\omega = e^{i(n\pi - \arg(x^*y))}$. This completes the first part of the proof.
	
	We next prove inequality~\eqref{eq:Herm}. In view of~\eqref{eq:scond1}, it is enough to prove
	\begin{equation}
		\sup\big \{ \big | \sum_{k=0}^d \lambda^k y^*\Delta_{A_k}x + \sum_{j=1}^m w_j(\lambda)y^*\Delta_{E_j}x \big | \big \} \geq \frac{1}{\sqrt{2}} \alpha,
	\end{equation}
	where $\alpha = \sum_{k=0}^d |\lambda|^k + \sum_{j=1}^m |w_j(\lambda)|$, and the supremum is taken over $(\Delta_{A_0},\ldots, \Delta_{A_d},\Delta_{E_1},\ldots,\Delta_{E_m})$ $\in (\text{Herm}(n))^{d+m+1}$ such that $\|\Delta_{A_k}\| \leq 1, \|\Delta_{E_j}\| \leq 1$. In view of Theorem~\ref{thm:setK}, this is equivalent to proving that
	\begin{equation}
		\sup_{\alpha_k,\beta_k,\delta_j,\gamma_j\in\R}\{| \sum_{k=0}^d \lambda^k (\alpha_k + i c\beta_k) + \sum_{j=1}^m w_j(\lambda)( \delta_j + ic \gamma_j)|~:~ \alpha_k^2+\beta_k^2\leq 1, \delta_j^2 + \gamma_j^2\leq 1 \} \geq \frac{1}{\sqrt{2}} \alpha.
	\end{equation}
 Let us define the set
	\begin{align*}
		S&:=\{ s : s= (\alpha_0,\alpha_1,\ldots,\alpha_d,\beta_0,\beta_1,\ldots,\beta_d, \delta_1,\ldots, \delta_m,~\gamma_1,\ldots, \gamma_m)\in\R^{2(d+m)+2} ,\\ 
		& \hspace{3cm} \alpha_k^2+\beta_k^2\leq 1 \text{ for } k=0,\ldots,d,~ \delta_j^2+\gamma_j^2\leq 1 \text{ for } j=1,\ldots, m \},
	\end{align*}
	and for a fixed choice of $s\in S$, we define
	\begin{equation}
		H_s(\lambda):= \sum_{k=0}^d \lambda^k  (\alpha_k + ic\beta_k) + \sum_{j=1}^m w_j(\lambda)(\delta_j + ic\gamma_j).
	\end{equation} 
	Further, denote by $\theta:= \arg(\lambda)$ and by $\theta_j:=\arg(w_j(\lambda))$, then the real and the complex part of $H_s(\lambda)$ can be expressed as
	\begin{align*}
		\real(H_s(\lambda)) &= \sum_{k=0}^d |\lambda|^k  (\cos (k\theta)\alpha_k - c \sin(k\theta) \beta_k) + \sum_{j=1}^m |w_j(\lambda)|(\cos(\theta_j)\delta_j - c\sin(\theta_j) \gamma_j),\\
		\imag(H_s(\lambda)) &= \sum_{k=0}^d |\lambda|^k  (\sin (k\theta)\alpha_k + c \cos(k\theta) \beta_k) + \sum_{j=1}^m |w_j(\lambda)|(\sin(\theta_j)\delta_j + c\cos(\theta_j) \gamma_j).
	\end{align*}
	Choosing $s_1\in S$ such that $\alpha_k=\text{sign}(\cos(k\theta)), \beta_k = 0$ for $k=0,1,\ldots,d$ and  $\delta_j=\text{sign}(\cos(\theta_j)), \gamma_j=0$ for $j=1,\ldots, m$, we obtain
	\begin{equation*}
		\real(H_{s_1}(\lambda)) = \sum_{k=0}^d |\lambda|^k |\cos(k\theta)| + \sum_{j=1}^m |w_j(\lambda)||\cos(\theta_j)|,
	\end{equation*} 
	while a choice of $s_2\in S$ with $\alpha_k=\text{sign}(\sin(k\theta)), \beta_k = 0$ for $k=0,1,\ldots,d$ and  $\delta_j=\text{sign}(\sin(\theta_j)), \gamma_j=0$ for $j=1,\ldots, m$ gives
	\begin{equation*}
		\imag(H_{s_2}(\lambda)) = \sum_{k=0}^d |\lambda|^k |\sin(k\theta)| + \sum_{j=1}^m |w_j(\lambda)||\sin(\theta_j)|.
	\end{equation*}
	Since, $\sup_{s\in S} |H_s(\lambda)| \geq \max\{\real(H_{s_1}(\lambda)),\imag(H_{s_2}(\lambda))\}$, this means
	\begin{align*}
		\sup_{s\in S} |H_s(\lambda)| & \geq \sum_{k=0}^d |\lambda|^k \max\{|\cos(k\theta)|,|\sin(k\theta)|\} + \sum_{j=1}^m |w_j(\lambda)| \max\{|\cos(\theta_j)|, |\sin(\theta_j)|\}\\
		& \geq \frac{1}{\sqrt{2}} \bigg( \sum_{k=0}^d |\lambda|^k + \sum_{j=1}^m |w_j(\lambda)|\bigg).
	\end{align*}
	This completes the proof.
\end{proof}

The structured eigenvalue condition number for $*$-even and $*$-odd matrix polynomials was considered in~\cite{Bor10}, and conditions on the eigenvalues were obtained under which the unstructured and the structured condition numbers are equal. Further, bounds for the structured condition number were also obtained. Here, we give a similar characterization of when the unstructured and structured condition numbers for the skew-Hermitian, $*$-even and $*$-odd structures, are equal and also some bounds for the structured condition number by using the already obtained characterization for the Hermitian condition number in Theorem~\ref{thm:herm}. This is followed by exploiting the fact that $G(z)$ is skew-Hermitian RMF if and only if $R(z)=iG(z)$ is Hermitian. Similarly, it is easy to observe from Table~\ref{tab:my_label}, that $G(z)$ is $*$-even or $*$-odd if and only if $R(z)=G(iz)$ or $R(z)=iG(iz)$ is Hermitian. To derive bounds for the structured condition number for $*$-even RMFs, we define
\begin{eqnarray*}
	\mathbb{S}_{e} \hspace{-.2cm}&:=\Big\{\left(\Delta_{A_{0}}, \ldots,\Delta_{A_{d}},\Delta_{E_{1}},\ldots, \Delta_{E_m}\right) \in (\C^{n\times n})^{d+m+1} :~ \Delta_{A_k}^*=(-1)^{k}\Delta_{A_k} ,\,\\
	&	k=0,\ldots, d,~\Delta_{E_j} \in {\rm Herm}(n),~ j=1,\ldots, m\Big\},
\end{eqnarray*}
when $(w_j(-z))^*=w_j(\overline z)$, and 
\begin{eqnarray*}
	\mathbb{S}_{e} \hspace{-.2cm}&:=\Big\{\left(\Delta_{A_{0}}, \ldots,\Delta_{A_{d}},\Delta_{E_{1}},\ldots, \Delta_{E_m}\right)  \in (\C^{n\times n})^{d+m+1}:~  \Delta_{A_k}^*=(-1)^{k}\Delta_{A_k},\\
	&	k=0,\ldots, d,~\Delta_{E_j} \in {\rm SHerm}(n),~ j=1,\ldots, m\Big\},
\end{eqnarray*}
when $(w_j(-z))^*=-w_j(\overline z)$. Similary, for the $*$-odd structure, we define
\begin{eqnarray*}
	\mathbb{S}_{o}\hspace{-.2cm}&:=\Big\{\left(\Delta_{A_{0}}, \ldots,\Delta_{A_{d}},\Delta_{E_{1}},\ldots, \Delta_{E_m}\right)  \in (\C^{n\times n})^{d+m+1}:~ \Delta_{A_k}^*=(-1)^{k+1}\Delta_{A_k},\\
	&	k=0,\ldots, d,~\Delta_{E_j} \in {\rm Herm}(n),~ j=1,\ldots, m\Big\},
\end{eqnarray*}
when $(w_j(-z))^*=-w_j(\overline z)$, and 
\begin{eqnarray*}
	\mathbb{S}_{o}\hspace{-.2cm}&:=\Big\{\left(\Delta_{A_{0}}, \ldots,\Delta_{A_{d}},\Delta_{E_{1}},\ldots, \Delta_{E_m}\right)  \in (\C^{n\times n})^{d+m+1}:~ \Delta_{A_k}^*=(-1)^{k+1}\Delta_{A_k},\\
	&	k=0,\ldots, d,~\Delta_{E_j} \in {\rm SHerm}(n),~ j=1,\ldots, m\Big\},
\end{eqnarray*}
when $(w_j(-z))^*=w_j(\overline z)$.
%
%\begin{align}
%	\S_e &:=\{ (\Delta_{A_0}, \Delta_{A_1},\ldots,  \Delta_{A_d}, \Delta_{E_1},\ldots,\Delta_{E_m})\in (\C^{n,n})^{d+m+1} : \Delta_{A_k}^* = (-1)^k\Delta_{A_k}, \, k=0,1,\ldots,d \nonumber \\
%	&\hspace{4cm} \Delta_{E_j} \in \text{Herm}(n) \text{ for $j$ s.t. } w_j(-\lambda)^*=-w_j(\overline{z}),\nonumber \\ & \hspace{4cm}\Delta_{E_j} \in \text{SHerm}(n) \text{ for $j$ s.t. } w_j(-\lambda)^*=w_j(\overline{z}) \}.
%\end{align}
%%
%Similarly, for $*$-odd structure we define
%%
%\begin{align}
%	\S_o &:=\{ (\Delta_{A_0}, \Delta_{A_1},\ldots,  \Delta_{A_d}, \Delta_{E_1},\ldots,\Delta_{E_m})\in (\C^{n,n})^{d+m+1} : \Delta_{A_k}^* = (-1)^{k+1}\Delta_{A_k}, \, k=0,1,\ldots,d \nonumber \\
%	&\hspace{4cm} \Delta_{E_j} \in \text{Herm}(n) \text{ for $j$ s.t. } w_j(-\lambda)^*=w_j(\overline{z}),\nonumber \\ & \hspace{4cm}\Delta_{E_j} \in \text{SHerm}(n) \text{ for $j$ s.t. } w_j(-\lambda)^*=-w_j(\overline{z}) \}.
%\end{align}
%
Then, from~\eqref{def:ucond}, the structured condition numbers are denoted by $\kappa^{\text{SHerm}}(\lambda,G)$ when $\S=(\text{SHerm}(n))^{d+m+1}$, $\kappa^{\text{even}_*}(\lambda,G)$ when $\S=\S_e$, and $\kappa^{\text{odd}_*}(\lambda,G)$ when $\S=S_o$. We have the following result.
\begin{theorem}\label{thm:herm1}
	Let $\mathbb S\in \left\{ {\mathbb S}_e,{\mathbb S}_o,({\rm SHerm})^{d+m+1} \right\}$ and let 
	$G(z)$ be a RMF of the form~\eqref{mat:G}, where the coefficient matrices $(A_0,\ldots,A_d,E_1,\ldots,E_m) \in \mathbb S$. Let $(\lambda,x,y)$ be an eigentriplet of $G(z)$ such that $\lambda$ satisfy the assumption~\eqref{assump}, and $\|x\|=1$ and $\|y\|=1$. Define $R(z)=G(iz)$. Then the following holds:
	\begin{itemize}
		\item when $\mathbb S= ({\rm SHerm})^{d+m+1}$, we have
		\[
		\kappa^{{\rm SHerm}}(\lambda,G)=\kappa^{{\rm Herm}}(\lambda,iG),
		\]
		\item when $\mathbb S= {\mathbb S}_e$, we have
		\[
		\kappa^{{\rm even}_*}(\lambda,G)=\kappa^{{\rm Herm}}(\frac{\lambda}{i},R),
		\]
		\item when $\mathbb S= {\mathbb S}_o$, we have
		\[
		\kappa^{{\rm odd}_*}(\lambda,G)=\kappa^{{\rm Herm}}(\frac{\lambda}{i},iR).
		\]
	\end{itemize}
\end{theorem}

As a direct consequence of Theorems~\ref{thm:herm} and~\ref{thm:herm1}, we obtain the following corollary.
\begin{corollary}
	Let $\mathbb S\in \left\{ {\mathbb S}_e,{\mathbb S}_o,({\rm SHerm})^{d+m+1} \right\}$ and let  $G(z)$ be a RMF of the form~\eqref{mat:G}, where the coefficient matrices $(A_0,\ldots,A_d,E_1,\ldots,E_m) \in \mathbb S$. Let 
$(\lambda,x,y)$ be an eigentriplet of $G(z)$ such that $\lambda$ satisfy the assumption~\eqref{assump}, and $\|x\|=1$ and $\|y\|=1$.
Then for $\lambda\neq 0$, $\kappa^{\S}(\lambda,G) = \kappa(\lambda,G)$ if and only if $x^*y = 0$ or $\lambda \in i\R$, and $\kappa^{\S}(0,G) = \kappa(0,G)$ if and only if $x^*y=0$ or $w_j(0) \in i\R$ for $j=1,\ldots,m$. Moreover, we have
	\begin{equation*}
		\frac{1}{\sqrt{2}}\kappa(\lambda,G) \leq \kappa^{\S}(\lambda,G).
	\end{equation*}
\end{corollary}

\subsection{T-even/T-odd structure}
In this section, we consider T-even and T-odd RMFs of the form~\eqref{mat:G} that satisfy $(G(z))^T=G(-z)$ or $(G(z))^T=-G(-z)$, respectively. This implies 
that for T-even RMFs $G(z)$, the coefficient matrices $A_k$
for $k=0,\ldots,d$ satisfy that $A_k^T=(-1)^kA_k$, and $E_j$ for $j=1,\ldots,m$ satisfy that, 
$E_j^T=E_j$ if $w_j(z)$ is an even function and $E_j^T=-E_j$ if $w_j(z)$ is an odd function. Similarly, for T-odd RMFs $G(z)$, the coefficient matrices $A_k$ for $k=0,\ldots,d$ satisfy that $A_k^T=(-1)^{k+1}A_k$, and
$E_j$ for $j=1,\ldots,m$ satisfy that, $E_j^T=E_j$ if $w_j(z)$ is an odd function and $E_j^T=-E_j$ if $w_j(z)$ is an even function, see Table~\ref{tab:my_label}.

To tackle the T-even/T-odd RMFs, we partition the set $X=\{1,\ldots,m\}$ into two subsets of indices $X_o:=\{p_1,\ldots,p_r\}$ and $X_e:=\{p_{r+1},\ldots,p_m\}$ with $p_i < p_j$ for $i<j$ such that 
$w_{p}(z)$ is an odd function if $p \in X_o$ and $w_{p}(z)$ is an even function if $p \in X_e$.
Thus, for T-even structures, we have $A_k^T=(-1)^kA_k$ for $k=0,\ldots,d$, and for  $E_j^T=E_j$ if $j \in X_e$ and $E_j^T=-E_j$ if $j \in X_o$.

Let us define two subsets of matrix tuples corresponding to T-even and T-odd structures by
\begin{eqnarray}\label{set:Teven}
	\mathbb{S}_e:&=\big \{ (A_0,\ldots, A_d, E_1,\ldots ,E_m) \in (\C^{n\times n})^{d+m+1}\; : \; A_k^T=(-1)^kA_k, \; k=0,\ldots, d,\nonumber\\
	& E_j^T=E_j\;\text{if}\; j\in X_e \;~ \text{and}\; ~E_j^T=-E_j \; \text{if}\;j \in X_o \big \},
\end{eqnarray}
and 
\begin{eqnarray}\label{set:Todd}
	\mathbb{S}_o:&=\big \{ (A_0,\ldots, A_d, E_1,\ldots ,E_m) \in (\C^{n\times n})^{d+m+1}\; : \; A_k^T=(-1)^{k+1}A_k, \; k=0,\ldots, d,\nonumber\\
	& E_j^T=E_j\;\text{if}\; j\in X_o \;~ \text{and}\; ~E_j^T=-E_j \; \text{if}\;j \in X_e  \big \},
\end{eqnarray}
and denote the T-even eigenvalue condition number by $\kappa^{{\rm even_T}}(\lambda,G)$ when $\mathbb S=\mathbb S_e$ in~\eqref{def:ucond}, and the T-odd eigenvalue condition number by $\kappa^{{\rm odd_T}}(\lambda,G)$ when $\mathbb S=\mathbb S_o$ in~\eqref{def:ucond}. In the rest of this section, we derive a result to estimate the T-even condition number $\kappa^{{\rm even_T}}(\lambda,G)$. A similar result for $\kappa^{{\rm odd_T}}(\lambda,G)$ can be obtained analogously.
\begin{theorem}
	Let $G(z)$ be a T-even RMF of the form~\eqref{mat:G} and let 
	$(\lambda,x,y)$ be an eigentriplet of $G(z)$ such that $\lambda$ satisfy the assumption~\eqref{assump}, and $\|x\|=1$ and $\|y\|=1$. Then, for $\lambda\not = 0$, $\kappa^{\rm even_T}(\lambda,G) = \kappa(\lambda,G)$ if and only if $x^Ty=0$, and $\kappa^{\rm even_T}(0,G) = \kappa(0,G)$ if and only if $x^Ty=0$ or $X_o=\emptyset$. Moreover, we have
	\begin{equation}\label{eq:Teven}
		\kappa^{\rm even_T}(\lambda,G) = \beta\cdot \kappa(\lambda,G),
	\end{equation}
	where 
	\begin{equation}
		\beta=\begin{cases}
			\frac{\sum_{k=0}^{\frac{d-1}{2}} \left( \left| \lambda \right|^{2k} + c\left| \lambda \right|^{2k+1} \right)
				+ \sum_{j\in X_e} \left| w_j(\lambda) \right| + \sum_{j\in X_o} c\left| w_j(\lambda) \right|}{\sum_{k=0}^d |\lambda|^k+\sum_{j=1}^m |\omega_j(\lambda)|} & \text{if } d \text{ is odd}, \\[10pt]
			\frac{1+\sum_{k=0}^{\frac{d}{2}} \left( \left| \lambda \right|^{2k} + c\left| \lambda \right|^{2k-1} \right)
				+ \sum_{j\in X_e} \left| w_j(\lambda) \right| + \sum_{j\in X_o} c\left| w_j(\lambda) \right|}{\sum_{k=0}^d |\lambda|^k+\sum_{j=1}^m |\omega_j(\lambda)|} & \text{if } d \text{ is even},
		\end{cases}
	\end{equation}
	with $c=\sqrt{1-|x^Ty|^2}$, and $X_o (\text{resp. }X_e)$ is the set of indices $j\in \{1,2,\ldots,m\}$ for which $w_j(\lambda)$ is an odd (even) function.
\end{theorem}
\begin{proof}
	First, suppose that $\lambda\not = 0$. Then, from Lemma~\ref{lem:SequalU}, $\kappa^{\rm even_T}(\lambda,G) = \kappa(\lambda,G)$ if and only if there exists $(\Delta_{A_0}, \ldots,\Delta_{A_d}, \Delta_{E_1}, \ldots,\Delta_{E_m}) \in \S_e$ such that $\|\Delta_{A_k}\|\leq 1$ and $\|\Delta_{E_j}\|\leq 1$ satisfying
	\begin{equation}\label{eq:Teven:map}
		\Delta_{A_k}x = \omega \text{sign}(\lambda^k)y,~k=0,\ldots,d \quad \text{and}\quad \Delta_{E_j}x = \omega \text{sign}(w_j(\lambda))y,~j=1,\ldots,m,
	\end{equation}
	where $\omega\in \C$ is such that $|\omega| = 1$ and $\text{sign}(z) = \frac{\overline{z}}{|z|}$. From the structured set $\mathcal{S}_e$ defined in~\eqref{set:Teven}, the perturbation matrix $\Delta_{A_k}$ is symmetric for even index $k$, and skew-symmetric for odd index $k$. $\Delta_{E_j}$ is symmetric for $j\in X_e$ and skew-symmetric for $j\in X_o$, where $X_o (\text{resp. }X_e)$ is the set of indices $j\in \{1,2,\ldots,m\}$ for which $w_j(\lambda)$ is an odd (even) function. On using Lemma~\ref{lem:map}, there always exists symmetric matrix $\Delta$ sending a vector $v$ to $u$, and there exists skew-symmetric matrix $\Delta$ sending a vector $v$ to $u$ if and only if $u^Tv=0$. Using this in~\eqref{eq:Teven:map}, there exists $\Delta_{A_k}$ for $k=0,1,\ldots,d$ and $\Delta_{E_j}$ for $j=1,\ldots, m$ satisfying~\eqref{eq:Teven:map} if and only if the following holds,
	\begin{align}
		\omega \text{sign}(\lambda^k) x^Ty = 0 ~& \text{ for odd } k\in \{0,1,\ldots, d\} \quad \text{and}\quad 
		\omega \text{sign}(w_j(\lambda))x^Ty = 0 & \text{ for } j\in X_o,
	\end{align}
	which collectively holds if and only if $x^Ty=0$. Similarly, when $\lambda=0$
	from Lemmas~\ref{lem:SequalU} and~\ref{lem:map}, we have that 
	 $\kappa^{\rm even_T}(0,G) = \kappa(\lambda,G)$ if and only if,
	\begin{equation}
		\omega \text{sign}(w_j(\lambda))x^Ty = 0 \quad \text{ for } j\in X_o,
	\end{equation}
	which holds if and only if $x^Ty=0$ or $X_o=\emptyset$. 
	
Next, we prove~\eqref{eq:Teven} only when the degree of the polynomial part of the RMF is odd, i.e., $d$ is odd. When $d$ is even, the result follows on similar lines. For this, let $\Delta G(z) = \sum_{k=0}^d z^k \Delta_{A_k} + \sum_{j=1}^m w_j(z) \Delta_{E_j}$ with $\|\Delta_{A_k}\| \leq 1, \|\Delta_{E_j}\| \leq 1$, then by Theorem~\ref{thm:setK}, we obtain 
	\begin{align}\label{eq:Teven1}
		|y^*\Delta G(\lambda)x| & = \big | \sum_{k=0}^{\frac{d-1}{2}} \left( \lambda^{2k} (\alpha_{2k} + i \beta_{2k}) + \lambda^{2k+1} c (\alpha_{2k+1} + i\beta_{2k+1}) \right) + \nonumber \\
		& \hspace{2cm} \sum_{j\in X_e} w_j(\lambda) (\delta_j +i\gamma_j) + \sum_{j\in X_o} w_j(\lambda) c(\delta_j +i\gamma_j) \big| \nonumber \\
		& \leq \sum_{k=0}^{\frac{d-1}{2}} \left( |\lambda|^{2k} + c|\lambda|^{2k+1} \right) + \sum_{j\in X_e} |w_j(\lambda)| + \sum_{j\in X_o} c |w_j(\lambda)|.
	\end{align} 
	Using~\eqref{eq:Teven1} in~\eqref{eq:scond1}, shows the inequality in~\eqref{eq:Teven}. For the equality in~\eqref{eq:Teven1}, we choose a symmetric matrix $S$ and a skew-symmetric matrix $R$ such that $y^*Sx=1$ and $y^*Rx=\sqrt{1-|x^Ty|^2}$. Such a choice is possible by Theorem~\ref{thm:setK}. The inequality in~\eqref{eq:Teven1} is attained for $\Delta G(z) = \sum_{k=0}^d z^k \Delta_{A_k} + \sum_{j=1}^m w_j(z) \Delta_{E_j}$, where $\Delta_{A_{2k}} = \text{sign}(z^{2k})S$ and $\Delta_{A_{2k+1}} = \text{sign}(z^{2k+1})R$ for $k=0,1,\ldots, \frac{d-1}{2}$, and $\Delta_{E_j} = \text{sign}(w_j(z))S$ for $j\in X_e$, $\Delta_{E_j} = \text{sign}(w_j(z))R$ for $j\in X_o$. This completes the proof.
\end{proof}

By following the lines similar to the T-even RMFs, the structured condition number for T-odd RMFs can also be obtained as stated in the following result. 

\begin{theorem}
	Let $G(z)$ be a T-odd RMF of the form~\eqref{mat:G} and let 
	$(\lambda,x,y)$ be an eigentriplet of $G(z)$ such that $\lambda$ satisfy the assumption~\eqref{assump}, and $\|x\|=1$ and $\|y\|=1$. Then, for $\lambda\not = 0$, $\kappa^{\rm odd_T}(\lambda,G) = \kappa(\lambda,G)$ if and only if $x^Ty=0$, and $\kappa^{\rm odd_T}(0,G) = \kappa(0,G)$ if and only if $x^Ty=0$ or $X_e=\emptyset$. Moreover, we have
	\begin{equation}\label{eq:Todd}
		\kappa^{\rm odd_T}(\lambda,G) = \gamma \cdot \kappa(\lambda,G),
	\end{equation}
	where 
	\begin{equation}
		\gamma=\begin{cases}
			\frac{\sum_{k=0}^{\frac{d-1}{2}} \left( \left| \lambda \right|^{2k} + c\left| \lambda \right|^{2k+1} \right)
				+ \sum_{j\in X_o} \left| w_j(\lambda) \right| + \sum_{j\in X_e} c\left| w_j(\lambda) \right|}{\sum_{k=0}^d |\lambda|^k+\sum_{j=1}^m |\omega_j(\lambda)|} & \text{if } d \text{ is odd}, \\[10pt]
			\frac{1+\sum_{k=0}^{\frac{d}{2}} \left( \left| \lambda \right|^{2k} + c\left| \lambda \right|^{2k-1} \right)
				+ \sum_{j\in X_o} \left| w_j(\lambda) \right| + \sum_{j\in X_e} c\left| w_j(\lambda) \right|}{\sum_{k=0}^d |\lambda|^k+\sum_{j=1}^m |\omega_j(\lambda)|} & \text{if } d \text{ is even},
		\end{cases}
	\end{equation}
	with $c=\sqrt{1-|x^Ty|^2}$, and $X_o (\text{resp. }X_e)$ is the set of indices $j\in \{1,2,\ldots,m\}$ for which $w_j(\lambda)$ is an odd (even) function.
\end{theorem}

\subsection{$*$-palindromic and T-palindromic structure}

In this section, we consider RMFs $G(z)$ of the form~\eqref{mat:G} with $*$-palindromic and $T$-palindromic structure from Table~\eqref{tab:my_label}. The RMF $G(z)$ is $*$-palindromic if the coefficient matrices $A_k$'s and $E_j$'s in $G(z)$ satisfies $A_k^*=A_{d-k}$ for $k=0,\ldots, d$ and $E_j^*=E_{j}$, and $w_j(z)^*=\overline{z}^dw_j(\frac{1}{\overline{z}})$ for $j=1,\ldots, m$. Similarly, $G(z)$ is $T$-palindromic if $A_k^T=A_{d-k}$ for $k=0,\ldots, d$ and $E_j^T=E_{j}$, and $w_j(z)={z}^dw_j(\frac{1}{{z}})$ for $j=1,\ldots, m$. To obtain the structured condition number for $*$-palindromic matrix functions $G(z)$, we define the following perturbation set
\begin{eqnarray}
	{\rm Pal}_* := &\Big\{ (\Delta_{A_{0}}, \ldots,\Delta_{A_{d}},\Delta_{E_{1}},\ldots, \Delta_{E_m}) \in (\C^{n\times n})^{d+m+1}:~(\Delta_{A_k})^*=\Delta_{A_{d-k}}, \nonumber\\ 
	& k=0,\ldots,d, \, \Delta_{E_j}^* =\Delta_{E_j},\, j=1,\ldots,m
	\Big\}.
\end{eqnarray}
and for $T$-palindromic structure
\begin{eqnarray}
	{\rm Pal}_T := &\Big\{ (\Delta_{A_{0}}, \ldots,\Delta_{A_{d}},\Delta_{E_{1}},\ldots, \Delta_{E_m}) \in (\C^{n\times n})^{d+m+1}:~(\Delta_{A_k})^T=\Delta_{A_{d-k}}, \nonumber\\ 
	& k=0,\ldots,d, \, \Delta_{E_j}^T =\Delta_{E_j},\, j=1,\ldots,m
	\Big\}.
\end{eqnarray}
Then, from~\eqref{def:ucond}, the structured condition number for $\mathbb{S} = {\rm Pal}_*$ and $\mathbb{S}={\rm Pal}_T$ are  respectively denoted by $\kappa^{\rm pal_*}(\lambda,G)$ and $\kappa^{\rm pal_T}(\lambda,G)$.
\begin{theorem}\label{thm:Tpal}
	Let $G(z)$ be a $T$-palindromic RMF of the form~\eqref{mat:G} and let 
	$(\lambda,x,y)$ be an eigentriplet of $G(z)$ such that $\lambda$ satisfy the assumption~\eqref{assump}, and $\|x\|=1$ and $\|y\|=1$. Then $\kappa^{\rm pal_T}(\lambda,G) = \kappa(\lambda,G)$ if and only if 
	\begin{equation*}
		 x^{T}y = 0 \quad \text{ or } \quad \lambda \in  \begin{cases} \mathbb{R} &\text{ if } d \text{ is even}\\
		 	\mathbb{R}^+ &\text{ if } d \text{ is odd},
		 \end{cases}
	\end{equation*}
	where $d$ is the degree of the polynomial part of the RMF $G(z)$. Moreover, we have 
	\begin{equation}\label{eq:Tpal}
		\frac{1}{\sqrt{2}} \frac{B_1(\lambda)}{|y^*G'(\lambda)x| } \leq \kappa^{\rm pal_T}(\lambda, G) \leq \frac{B_1(\lambda)}{|y^*G'(\lambda)x|},
	\end{equation}
	where
	\begin{equation}\label{def:b1lam}
		B_1(\lambda) =
		\begin{cases}
			\sum_{k=0}^{\frac{d-1}{2}} |\lambda^k + \lambda^{d-k}| + c|\lambda^k - \lambda^{d-k}|  +\sum_{j=1}^{m} |w_j(\lambda)| & \text{if d is odd}\\
			\sum_{k=0}^{\frac{d-2}{2}} |\lambda^k + \lambda^{d-k}| + c|\lambda^k - \lambda^{d-k}| + |\lambda|^\frac{d}{2}+\sum_{j=1}^{m} |w_j(\lambda)| & \text{if d is even},
		\end{cases}
	\end{equation}
	and $c=\sqrt{1-|x^Ty|^2}$.
\end{theorem}
\begin{proof}\sloppy
	First suppose that $\lambda \neq 0$. Then by Lemma~\ref{lem:SequalU}, $\kappa^{\rm pal_T}(\lambda,G) = \kappa(\lambda,G)$ if and only if there exists $(\Delta_{A_0},\ldots,\Delta_{A_d},\Delta_{E_1},\ldots,\Delta_{E_m})\in {\rm Pal}_T$ satisfying
	\begin{align}\label{eq:Tpal:map}
		\Delta_{A_k}x & = \omega \text{sign}(\lambda^k)y~ \text{ for } k=0,\ldots,d \quad \text{and}\quad 
		\Delta_{E_j}x & = \omega \text{sign}(w_j(\lambda))y~ \text{ for } j=1,\ldots,m,
	\end{align}
	where $\omega \in \C$ with $|\omega|=1$. 
In view of Lemma~\ref{lem:map}, there always exist symmetric matrices $\Delta_{E_j}$ for $j=1,\ldots,m$ satisfying~\eqref{eq:Tpal:map}. Since $G(z)$ is T-palindromic, the first constraint in~\eqref{eq:Tpal:map} is equivalent to
	\begin{equation}\label{eq:Tpal:mapA}
		\Delta_{A_k}x = \omega \text{sign}({\lambda^k})y \quad \text{ and } \quad \Delta_{A_k}^Tx = \omega \text{sign}({\lambda^{d-k}})y.
	\end{equation} 
From Lemma~\ref{lem:map-pal}, there exists matrix $\Delta_{A_k} \in \C^{n,n}$ satisfying~\eqref{eq:Tpal:mapA} if and only if $\omega \text{sign}(\lambda^k)x^Ty = \omega \text{sign}(\lambda^{d-k})y^Tx$ for $k=0,\ldots, d$. The later condition holds if and only if $x^Ty=0$ or sign$(\lambda^{d-2k}) =1$. The condition sign$(\lambda^{d-2k}) =1$ implies that $\lambda\in \R$ if $d$ is even and $\lambda\in \R^+$ if $d$ is odd.
When $\lambda=0$, then it is easy to verify using Lemma~\ref{lem:SequalU} that  $\kappa^{\rm pal_T}(0,G) = \kappa(0,G)$ always holds. This completes the first part. 
	
Next, we prove~\eqref{eq:Tpal}. For this, we only consider the case when the degree of the polynomial part of the RMF $G(z)$ is even, i.e., $d$ is even. The case when $d$ is odd can be obtained analogously. In view of~\eqref{eq:scond1}, it is enough to prove 
	\begin{equation}\label{eq:Tpal1}
		\frac{1}{\sqrt{2}} B_1(\lambda) \leq \sup \big |\sum_{k=0}^d \lambda^k y^*\Delta_{A_k}x + \sum_{j=1}^m w_j(\lambda)y^*\Delta_{E_j}x\big | \leq B_1(\lambda),
	\end{equation}
	where $B_1(\lambda)$ is defined by~\eqref{def:b1lam} and supremum is taken over 
	$(\Delta_{A_0},\ldots,\Delta_{A_d},\Delta_{E_1},\ldots,\Delta_{E_m})\in {{\rm Pal}_T}$ such that $\|\Delta_{A_k}\|\leq 1, \|\Delta_{E_j}\| \leq 1$. The term $\sum_{k=0}^d \lambda^k y^*\Delta_{A_k}x$ in~\eqref{eq:Tpal1} can be rearranged as
	\begin{align*}
		\sum_{k=0}^{d} \lambda^k y^* \Delta_{A_k}x &= \sum_{k=0}^{\frac{d-2}{2}} \big (\lambda^k y^* \Delta_{A_k}x + \lambda^{d-k} y^* \Delta_{A_{k}}^T x  \big)+ \lambda^{\frac{d}{2}} y^* \Delta_{A_{\frac{d}{2}}} x \\ 
		& = \sum_{k=0}^{\frac{d-2}{2}} \big( (\lambda^k + \lambda^{d-k}) y^* \frac{\Delta_{A_k} + \Delta_{A_k}^T}{2} x + (\lambda^k - \lambda^{d-k}) y^* \frac{\Delta_{A_k} - \Delta_{A_k}^T}{2} x  \big ) + \lambda^{\frac{d}{2}} y^* \Delta_{A_{\frac{d}{2}}} x .
	\end{align*}
The matrices $\frac{\Delta_{A_k} + \Delta_{A_k}^T}{2}$ for $k=0,\ldots, \frac{d-2}{2}$, $ \Delta_{A_{\frac{d}{2}}}$, and $\Delta_{E_j}$ for $j=1,\ldots, m$ are symmetric, and $\frac{\Delta_{A_k} - \Delta_{A_k}^T}{2}$ for $k=0,\ldots, \frac{d-2}{2}$ are skew-symmetric. Therefore, from Theorem~\ref{thm:setK}, we obtain
	\begin{align}\label{eq:Tpal2}
		\sup \big |\sum_{k=0}^d \lambda^k y^*\Delta_{A_k}x +& \sum_{j=1}^m w_j(\lambda)y^*\Delta_{E_j}x \big|  = \sup | \sum_{k=0}^{\frac{d-2}{2}} \big((\lambda^k + \lambda^{d-k}) (\alpha_k + i \beta_k)  \nonumber \\ 
		& \hspace{-1cm}+ (\lambda^k - \lambda^{d-k}) c(\alpha_{\frac{d}{2}+k+1} + i \beta_{\frac{d}{2}+k+1})\big)+ \lambda^{\frac{d}{2}} (\alpha_{\frac{d}{2}} + i \beta_{\frac{d}{2}}) + \sum_{j=1}^m w_j(\lambda)(\delta_j +i \gamma_j)\big |,
	\end{align}
	where $\alpha_k,\beta_k,\delta_j,\gamma_j \in \R$ such that $\alpha_k^2 + \beta_k^2 \leq 1$ for $k=0,\ldots,d$ and $\delta_j^2 + \gamma_j^2 \leq 1$ for $j=1,\ldots, m$. The upper bound in~\eqref{eq:Tpal1} is immediate from~\eqref{eq:Tpal2} using the triangle inequality. For the lower bound, take $\theta:= \arg(\lambda), \theta_j:=\arg(w_j(\lambda)), \phi:= \arg(\lambda^k + \lambda^{d-k}), \hat \phi:= \arg(\lambda^d - \lambda^{d-k})$. Then~\eqref{eq:Tpal2} can be written as 
	\begin{align}\label{eq:Tpal3}
		\sum_{k=0}^d \lambda^k y^*\Delta_{A_k}x + \sum_{j=1}^m w_j(\lambda)y^*\Delta_{E_j}x & = \sum_{k=0}^{\frac{d-2}{2}} \bigg( |\lambda^k + \lambda^{d-k}|(\cos(\phi) + i\sin(\phi)) (\alpha_k + i \beta_k)\nonumber \\
		& \hspace{-2cm} + |\lambda^k - \lambda^{d-k}|(\cos(\hat \phi) + i \sin(\hat \phi)) c(\alpha_{\frac{d}{2}+k+1} + i \beta_{\frac{d}{2}+k+1}) \bigg)\nonumber \\ 
		&\hspace{-3cm} + |\lambda^{\frac{d}{2}}| (\cos(\frac{d\theta}{2}) + i\sin(\frac{d\theta}{2})) (\alpha_{\frac{d}{2}} + i \beta_{\frac{d}{2}}) + \sum_{j=1}^m |w_j(\lambda)| (\cos(\theta_j) + i\sin(\theta_j))(\delta_j +i \gamma_j).
	\end{align}
	Separating~\eqref{eq:Tpal3} into real and imaginary parts and following the arguments similar to the proof of Theorem~\ref{thm:herm}, we obtain that
	\begin{align*}
		|\sum_{k=0}^d \lambda^k y^*\Delta_{A_k}x + \sum_{j=1}^m w_j(\lambda)y^*\Delta_{E_j}x| & \geq \sum_{k=0}^{\frac{d-2}{2}}\bigg( \max\{ |\cos(\phi)|, |\sin(\phi)|\} |\lambda^k + \lambda^{d-k}| \nonumber \\ 
		&\hspace{-2cm} + c \max\{ |\cos(\hat \phi)|, |\sin(\hat \phi)|\} |\lambda^k - \lambda^{d-k}| \bigg) \nonumber \\
		&\hspace{-3cm} + \max\{ |\cos(\frac{d\theta}{2})|,|\sin(\frac{d\theta}{2})| \} |\lambda^{\frac{d}{2}}| + \sum_{j=1}^m \max\{ |\cos(\theta_j)|,|\sin(\theta_j)| \} |w_j(\lambda)| \nonumber \\
		& \geq \frac{1}{\sqrt{2}}\bigg( \sum_{k=0}^{\frac{d - 2}{2}} |\lambda^k + \lambda^{d-k}| + c |\lambda^k - \lambda^{d-k}| + | \lambda^{\frac{d }{2}}| + \sum_{j=1}^m |w_j(\lambda)|\bigg).
	\end{align*}
	This completes the proof.
\end{proof}

%\begin{remark}
%	{\color{red} Compare structured condition number with unstructured for T-palindromic RMF (maybe with an example).}
%\end{remark}

\begin{theorem}\label{thm:*pal}
	Let $G(z)$ be a $*$-palindromic RMF of the form~\eqref{mat:G} and let 
	$(\lambda,x,y)$ be an eigentriplet of $G(z)$ such that $\lambda \neq 0$ satisfy the assumption~\eqref{assump}, and $\|x\|=1$ and $\|y\|=1$. Then $\kappa^{\rm pal_*}(\lambda,G) = \kappa(\lambda,G)$ if and only if $x^*y=0$ or 
	\begin{equation*}
		\arg(w_j(\lambda)) - \frac{d}{2} \arg(\lambda) \in \pi \mathbb{Z},
	\end{equation*}
	where $d$ is the degree of the polynomial part of the RMF $G(z)$. Moreover, we have 
	\begin{equation}\label{eq:*pal}
		\frac{1}{\sqrt{2}} \frac{B_2(\lambda)}{|y^*G'(\lambda)x| } \leq \kappa^{\rm pal_*}(\lambda, G),
	\end{equation}
	where
	\begin{equation}\label{eq:defb2lam}
		B_2(\lambda) =
		\begin{cases}
			\sum_{k=0}^{\frac{d-1}{2}} |\lambda^k + \lambda^{d-k}| + |\lambda^k - \lambda^{d-k}|  +\sum_{j=1}^{m} |w_j(\lambda)| & \text{if d is odd}\\
			\sum_{k=0}^{\frac{d-2}{2}} \big(|\lambda^k + \lambda^{d-k}| + |\lambda^k - \lambda^{d-k}| \big)+ |\lambda|^\frac{d}{2}+\sum_{j=1}^{m} |w_j(\lambda)| & \text{if d is even}.
		\end{cases}
	\end{equation}
\end{theorem}
\begin{proof}\sloppy
First suppose that $\lambda\neq 0$. Then by Lemma~\ref{lem:SequalU}, $\kappa^{\rm pal_*}(\lambda,G) = \kappa(\lambda,G)$ if and only if there exists $(\Delta_{A_0},\ldots,\Delta_{A_d},\Delta_{E_1},\ldots,\Delta_{E_m})\in {\rm Pal}_*$ satisfying
	\begin{align}\label{eq:*pal:map1}
		\Delta_{A_k}x  = \omega \text{sign}(\lambda^k)y~ \text{ for } k=0,\ldots,d 
		\quad \text{and}\quad
		\Delta_{E_j}x  = \omega \text{sign}(w_j(\lambda))y~ \text{ for } j=1,\ldots,m.
	\end{align}
	The first constraint in~\eqref{eq:*pal:map1} is equivalent to
	\begin{equation*}\label{eq:*pal:mapA}
		\Delta_{A_k}x = \omega \text{sign}({\lambda^k})y \quad \text{ and } \quad \Delta_{A_k}^*x = \omega \text{sign}({\lambda^{d-k}})y.
	\end{equation*} 
	Using Lemmas~\ref{lem:map} and~\ref{lem:map-pal}, there exist matrices $\Delta_{A_k} \in \C^{n,n}$ for $k=0,\ldots,d$ and Hermitian $\Delta_{E_j}$ for $j=1,\ldots,m$ satisfying~\eqref{eq:*pal:map1} if and only if
	\begin{subequations}
		\begin{alignat}{2}
		&	\omega \text{sign}(\lambda^k)x^*y =\overline{ \omega \text{sign}(\lambda^{d-k})}y^*x \quad \text{for}~ ~k=0,\ldots,d \label{eq:*pal1A}\\
		&	\omega \text{sign}(w_j(\lambda)) x^*y\in \R \quad \text{for}~ ~j=1,\ldots,m.\label{eq:*pal1E}
		\end{alignat}
	\end{subequations}
	It is easy to verify that~\eqref{eq:*pal1A} holds if and only if $x^*y=0$ or $\arg(\omega) + \arg(x^*y) = \frac{d}{2} \arg(\lambda),$
	and~\eqref{eq:*pal1E} holds if and only if $x^*y=0$ or $\arg(\omega) - \arg(w_j(\lambda)) + \arg(x^*y) \in \pi \mathbb{Z} \quad \text{for } j=1,\ldots, m.$
	Combining these two conditions yields that $\kappa^{\rm pal_*}(\lambda,G) = \kappa(\lambda,G)$ if and only if $x^*y=0$ or $\arg(w_j(\lambda)) - \frac{d}{2} \arg(\lambda) \in \pi \mathbb{Z}$ for $j=1,\ldots, m$. 
	
	We now prove~\eqref{eq:*pal}, when the degree of the polynomial part of the RMF $G(z)$ is even, i.e., $d$ is even. The proof for the case when $d$ is odd follows similarly. In view of~\eqref{eq:scond1}, it is enough to prove 
	\begin{equation}\label{eq:*pal1}
		\frac{1}{\sqrt{2}} B_2(\lambda) \leq \sup |\sum_{k=0}^d \lambda^k y^*\Delta_{A_k}x + \sum_{j=1}^m w_j(\lambda)y^*\Delta_{E_j}x| \leq  B_2(\lambda),
	\end{equation}
	where $B_2(\lambda)$ is defined by~\eqref{eq:defb2lam} and supremum is taken over 
	$(\Delta_{A_0},\ldots,\Delta_{A_d},\Delta_{E_1},\ldots,\Delta_{E_m})\in {{\rm Pal}_*}$ such that $\|\Delta_{A_k}\|\leq 1, \|\Delta_{E_j}\| \leq 1$. The proof of~\eqref{eq:*pal1} is analgous to the proof of~\eqref{eq:Tpal1} in Theorem~\ref{thm:Tpal}. 
\end{proof}
%
%\begin{remark}
We note that for any $\lambda\in \C $, we have
	\begin{equation*}
		\sum_{k=0}^d |\lambda|^k \leq \begin{cases}
			\sum_{k=0}^{\frac{d-2}{2}} |\lambda^d + \lambda^{d-k}| + |\lambda^d - \lambda^{d-k}| + |\lambda^{\frac{d}{2}}| & \text{ if } d \text{ is even}\\
			\sum_{k=0}^{\frac{d-2}{2}} |\lambda^d + \lambda^{d-k}| + |\lambda^d - \lambda^{d-k}| & \text{ if } d \text{ is odd}.
		\end{cases} 
	\end{equation*}
	This directly gives from Theorems~\ref{thm:unstrcond} and~\ref{thm:*pal} that 
	\begin{equation*}
		\frac{1}{\sqrt{2}} \kappa(\lambda,G) \leq \frac{1}{\sqrt{2}} \frac{B_2(\lambda)}{|y^*G'(\lambda)x|} \leq \kappa^{\rm pal_*}(\lambda,G) \leq \kappa(\lambda,G).
	\end{equation*}
%\end{remark}

\section{Conclusion}
We have extended the framework presented in~\cite{Bor10} to compute the structured condition number of RMFs $G(z)$ of the form~\eqref{mat:G} for structures from  Table~\ref{tab:my_label}. We have derived the necessary and sufficient conditions for each structure considered, under which unstructured and structured condition numbers are equal. We have also obtained exact expression for the structured condition number of simple eigenvalues of symmetric, skew-symmetric, and T-even/odd rational matrix functions, and established tight bounds for Hermitian, skew-Hermitian, $*$-even/odd, $*$-palindromic and T-palindromic rational matrix functions. To our knowledge, no other work has been done in the literature to compute the condition number of RMFs with symmetry structures. 

Future research may involve computing the condition number under general perturbations in $G(z)$, i.e., when the scalar functions $w_j$'s in~\eqref{mat:delG} are also perturbed. Another possibility may involve examining other forms, such as the system matrix associated with the RMF, to compute the structured condition number.

\bibliographystyle{spmpsci}
\bibliography{RitPS25}

\end{document}